\documentclass[11pt]{amsart}

\usepackage{amssymb,verbatim}
\usepackage{amsmath,amsfonts}
\usepackage[mathscr]{euscript} 
\usepackage{amsthm}
\usepackage{url}
\usepackage{graphicx} 
\usepackage{enumitem} 
\usepackage{hyperref} 
\hypersetup{colorlinks} 
\usepackage{bbm} 
\usepackage{bm} 
\usepackage{mathrsfs} 
\usepackage{cancel} 

\usepackage[colorinlistoftodos]{todonotes}

\RequirePackage{cleveref}
\usepackage{hypcap}
\hypersetup{colorlinks=true, citecolor=darkblue, linkcolor=darkblue}
\definecolor{darkblue}{rgb}{0.0,0,0.7}
\newcommand{\darkblue}{\color{darkblue}}

\definecolor{darkred}{rgb}{0.68,0,0}
\newcommand{\darkred}{\color{darkred}}

\definecolor{darkgreen}{rgb}{0,.38,0}
\newcommand{\darkgreen}{\color{darkgreen}}

\newcommand{\defn}[1]{\emph{\darkblue #1}}
\newcommand{\defna}[1]{\emph{\darkred #1}}
\newcommand{\defnb}[1]{\emph{\darkblue #1}}
\newcommand{\defng}[1]{\emph{\darkgreen #1}}

\setlist[enumerate]{
	label=\textnormal{({\roman*})},
	ref={\roman*}}

\makeatletter
\def\th@plain{%
	\thm@notefont{}
	\itshape 
}
\def\th@definition{%
	\thm@notefont{}
	\normalfont 
}
\makeatother

\newtheorem{thm}{Theorem}[section]
\newtheorem{lemma}[thm]{Lemma}

\newtheorem*{claim*}{Claim}
\newtheorem{cor}[thm]{Corollary}
\newtheorem{prop}[thm]{Proposition}
\newtheorem{conj}[thm]{Conjecture}
\newtheorem{conv}[thm]{Convention}

\newtheorem{op}[thm]{Open Problem}

\theoremstyle{definition}
\newtheorem{ex}[thm]{Example}
\newtheorem{rem}[thm]{Remark}

\numberwithin{figure}{section}
\numberwithin{equation}{section}

\def\wh{\widehat}

\def\emp{\nothing}

\def\zz{\mathbb Z}
\def\nn{\mathbb N}

\def\ll{{\mathcal{L}}}
\def\rr{\mathbb R}

\def\pp{\mathbb P}

\def\ov{\overline}

\def\sm{\smallsetminus}

\def\Om{\Omega}
\def\BOm{{\mathbf{\Omega}}}

\def\la{\lambda}
\def\ga{\gamma}
\def\si{\sigma}
\def\de{\delta}

\def\al{\alpha}
\def\be{\beta}

\def\ve{\varepsilon}

\def\vp{\varphi}

\def\cC{\mathcal C}
\def\cB{\mathcal B}
\def\cI{\mathcal I}

\def\cA{\mathcal A}

\def\ce{\mathcal E}

\def\cH{\mathcal H}

\def\cR{\mathcal R}
\def\cS{\mathcal S}

\def\ssu{\subset}

\def\<{\langle}
\def\>{\rangle}

\def\0{{\mathbf 0}}

\def\nothing{\varnothing}

\def\.{\hskip.06cm}
\def\ts{\hskip.03cm}

\def\bz{{\textbf{z}}}

\def\bw{\textbf{\textit{w}}}

\def\bv{\textbf{\textit{v}}}
\def\ba{\textbf{\textbf{a}}}

\def\bc{\textbf{\textbf{c}}}
\def\bct{\emph{\textbf{\textbf{c}}}}

\def\di{{\small{\ts\diamond\ts}}}

\def\ze{{\zeta}}

\newcommand{\Ehr}{\mathrm{Ehr}}

\newcommand{\maj}{\mathrm{maj}}

\def\aN{\textrm{N}}
\def\aNr{\textrm{\em N}}

\def\.{\hskip.06cm}
\def\ts{\hskip.03cm}

\def\nin{\noindent}

\def\SP{{\textsc{\#P}}}

\def\poly{{\textsc{P}}}

\def\GapP{{\textsc{GapP}}}

\def\xx{\textbf{\textit{x}}}
\def\yy{\textbf{\textit{y}}}
\def\uu{\textbf{\textit{u}}}
\def\aa{\textbf{{\textit{a}}}}

\def\RG{\textrm{G}}
\def\RH{\textrm{H}}
\def\GG{\wh{\RG}}

\newcommand{\inv}{\operatorname{{\rm inv}}}

\newcommand{\LE}{\ce}

\def\Com{ {\text {\rm Com} } }

\DeclareMathOperator{\Ec}{\mathcal{E}} 

\DeclareMathOperator{\Rb}{\mathbb{R}} 
\def\precc{\preccurlyeq} 
\title{Effective poset inequalities}
\date{\today}

\author{Swee Hong Chan}
\address[Swee Hong Chan]{Department of Mathematics, Rutgers University,  Piscatway, NJ 08854.}
\email{\texttt{sc2518@rutgers.edu}}

\author[\ts Igor Pak]{Igor Pak}
\address[Igor Pak]{Department of Mathematics, UCLA,  Los Angeles, CA 90095.}
\email{\texttt{pak@math.ucla.edu}}

\author[\ts Greta Panova]{Greta Panova}
\address[Greta Panova]{Department of Mathematics, USC,  Los Angeles, CA 90089.}
\email{\texttt{gpanova@usc.edu}}

\begin{document}

\begin{abstract}
We prove a number of new inequalities for the numbers of
linear extensions and order polynomials of finite posets.
First, we generalize the \emph{Bj\"orner--Wachs inequality}
to inequalities on order polynomials and their $q$-analogues
via direct injections and FKG inequalities, and 
establish several new inequalities on order polynomials. 

Second, we generalize actions of Coxeter groups on
restricted linear extensions, leading to vanishing and uniqueness
conditions for the \emph{generalized Stanley inequality}.
Third, we generalize the \emph{Sidorenko inequality}
to posets with small chain intersections and give complexity
theoretic applications.
\end{abstract}
	
\maketitle
	

\section{Introduction}

\medskip

\subsection{Foreword}\label{ss:intro-foreword}
There are two schools of thought on what to do when an \defng{interesting
combinatorial inequality} is established.  The first approach would
be to treat it as a tool to prove a desired result.  The inequality
can still be sharpened or generalized as needed, but this effort is
aimed with applications as the goal and not about the inequality per se.

The second approach is to treat the inequality as a result of importance
in its own right.  The emphasis then shifts to finding the ``right proof''
in an attempt to understand, refine or generalize it, in which case
we say that the inequality can be made \defna{effective}.  This is where the
nature of the inequality intervenes --- when both sides count combinatorial 
objects, the desire to relate these objects is overpowering.   

The inequality can be made effective in several different ways.   
A \defng{direct injection} \ts can give it a combinatorial interpretation
for the difference or prove the equality conditions.
Such an injection can also be a work of art,
inspiring and thought-provoking in the best case.
Alternatively, a \defng{technical proof} (say, probabilistic or algebraic),
can establish tools for generalizations out of reach by direct combinatorial
arguments.

Both types of proof are most impactful when presented in combination.
Making comparisons between different approaches can lead to further
results, new open problems, and is the source of wonder of the beauty
and diversity of mathematics.

\medskip

As the reader must have guessed, we aim to make effective several
celebrated combinatorial inequalities for the numbers of
linear extensions of finite posets:

\smallskip

\qquad $\circ$ \ the \defn{Bj\"orner--Wachs inequality},

\smallskip

\qquad $\circ$ \ the \defn{Sidorenko inequality}, and

\smallskip

\qquad $\circ$ \ the \defn{generalized Stanley inequality}.

\smallskip

\nin
Although there is a certain commonality of tools and approaches,
our investigation of these inequalities are largely independent,
united by the goal of being effective, i.e.\ extending these 
inequalities with the goal of understanding them on a deeper level.  
In addition to injections, we also use \defng{probabilistic} \ts 
and \defng{algebraic tools}, with some curious combinatorial twists.

\smallskip

\subsection{Extensions and generalizations of the Bj\"orner--Wachs inequality}\label{ss:intro-HP}

Let \. $P=(X,\prec)$ \. be a poset with \. $|X|=n$ \. elements.
For each element \ts $x \in X$, let \.
$B(x) := \big\{y \in X \. : \. y \succcurlyeq x  \big\}$ \.
be the \defn{upper order ideal} generated by~$x$, and let \. $b(x):=|B(x)|$.

A \defn{linear extension} of $P$ is a bijection \. $f: X \to [n]=\{1,\ldots,n\}$,
such that
\. $f(x) < f(y)$ \. for all \. $x \prec y$.
Denote by \ts $\Ec(P)$ \ts the set of linear extensions of $P$,
and let \. $e(P):=|\Ec(P)|$.

\smallskip

\begin{thm}[{\rm Bj\"orner and Wachs~\cite[Thm~6.3]{BW89}}{}]\label{t:HP}
Let \ts $P=(X,\prec)$ \ts be a poset with \ts $|X|=n$ \ts elements.
In the notation above, we have:
\begin{equation}\label{eq:HP}
e(P) \ \geq  \  \. n! \.\cdot \. \prod_{x \in X} \, \frac{1}{b(x)}\,.
\end{equation}
\end{thm}

\smallskip

This inequality was popularized by Stanley who stated it without proof
or a reference in \cite[Exc.~3.57]{Sta-EC}.\footnote{Richard Stanley
informed us that he indeed took it from~\cite{BW89} (personal communication, March 27, 2022).}
When the poset is a tree rooted at the
minimal element, the inequality in the theorem is an equality known
in the literature as the \defn{hook-length formula for trees}.  A
variation on the classical \defn{hook-length formulas} for straight and
shifted Young diagrams, this case is usually attributed to Don Knuth (1973),
see e.g.~\cite{Bea,SY}.  Although for other families of poset the lower bound
on~$e(P)$ given by~\eqref{eq:HP} is relatively weak, nothing better is
known in full generality, see e.g.~\cite{BP,MPP-asy,Pak-case}.

\smallskip

We start by recalling the original direct injective proof by Bj\"orner and Wachs
of the inequality~\eqref{eq:HP}. This allows us to prove that the inequality is in~$\SP$
(Theorem~\ref{t:HP-SP}).  We then obtain the following extension
of Theorem~\ref{t:HP}.

\smallskip

Let \ts $[k]:=\{1,\ldots,k\}$.
For an integer \ts $t\ge 1$, denote by \. $\Omega(P,t)$ \. the number of
order preserving maps \. $g: X\to [t]$, i.e.\ maps which satisfy
\. $g(x)\le g(y)$ \. for all \. $x \prec y$.  This is the \defn{order polynomial}
corresponding to poset~$P$, see e.g.~\cite[$\S$3.12]{Sta-EC}.

\smallskip

\begin{thm}\label{thm:HP order poly}
Let \ts $P=(X,\prec)$ \ts be a poset with \ts $|X|=n$ \ts elements.  Then,
for every \ts $t\in \nn$, we have:
\begin{equation}\label{eq:OP-HP}
\Omega(P,t) \ \geq \ t^r \ts (t+1)^{n-r} \, \prod_{x \in X} \. \frac{1}{b(x)}\,,
\end{equation}
where $r$ is the number of maximal elements of~$P$.
\end{thm}

\smallskip

Let us note that
\begin{equation}\label{eq:OP-asy}
\Omega(P,t) \ \sim \ \frac{e(P)\, t^n}{n!}   \quad \text{as} \ \ \ t \to \infty\ts.
\end{equation}
Thus, Theorem~\ref{thm:HP order poly} implies Theorem~\ref{t:HP}.
Note also that \. $\Omega(P,t) \ts = \ts \Omega(P^\ast,t)$, where \ts $P^\ast=(X, \prec^\ast)$ \ts
is the poset where relations are reversed:  \ts $x\prec y \ \Leftrightarrow \ y \prec^\ast x$.
Thus, the theorem holds when maximal elements are replaced with minimal elements.

The tools we use to establish Theorem~\ref{thm:HP order poly} are based on
\defng{Shepp's lattice}, and are extremely
far-reaching.  Notably, they allows us to establish 
the following \defng{strict log-concavity} \ts of the
order polynomial:

\smallskip

\begin{thm}[{\rm = Theorem~\ref{thm:log-concave-strict}}{}]
\label{thm:log-concave-intro}
	Let $P=(X,\prec)$ be a finite poset.
	Then, for every integer \ts $t \geq 2$, we have:
	\[ \Omega(P,t)^2  \  > \ \Omega(P,t+1) \. \cdot\.  \Omega(P,t-1).\]
\end{thm}

\smallskip

We use this result to we obtain the asymptotic version of
\defng{Graham's conjecture} proved by Daykin--Daykin--Paterson in 
\cite{DDP} by a direct injective argument (Theorem~\ref{conj:Graham}).   
%
%
%
%
%
Our next result is a general lower bound on the order polynomial strengthening the
asymptotic formula~\eqref{eq:OP-asy}.

\smallskip

\begin{thm}\label{thm:order_le}
Let \ts $P=(X,\prec)$ \ts be a poset with \ts $|X|=n$ \ts elements.  Then,
for every \ts $t\in \nn$, we have:
\begin{equation}\label{eq:OP-gen}
\Omega(P,t) \ \geq \ \frac{e(P)\, t^n}{n!}\..
\end{equation}
\end{thm}

\smallskip

Our proof of Theorem~\ref{thm:order_le} uses a direct injection.  Among other
applications of this approach, we prove that~\eqref{eq:OP-gen} is an equality
if and only if \ts $P$ \ts is an antichain (Corollary~\ref{c:OP2-equality}).

\smallskip

Since the Bj\"orner--Wachs inequality~\eqref{eq:HP} can be rather weak
in various special cases, neither of Theorems~\ref{thm:HP order poly} and~\ref{thm:order_le}
implies another (see Example~\ref{e:OP-Stanley}).
In a different direction, inequality~\eqref{eq:OP-gen}
strengthens the trivial inequality
\begin{equation}\label{eq:OP-gen-easy}
\Omega(P,t) \ \geq \ e(P) \cdot \binom{t}{n}
 \ = \ \frac{e(P)\,\. t\ts (t-1)\cdots (t-n+1)}{n!}\..
\end{equation}
Here the RHS counts the number of injections \. $f: X\to [t]$,
which are naturally mapped onto~$\Ec(P)$.  Note that~\eqref{eq:OP-gen}
agrees with~\eqref{eq:OP-gen-easy} in the leading term given also by~\eqref{eq:OP-asy},
but is sharper in the second term of the asymptotics.

Finally, we include an unpublished remarkably simple proof of the Bj\"orner--Wachs
inequality by Vic Reiner, via extension of the inequality to its $q$-analogue
(Theorem~\ref{t:HP-q}).  We then use our tools in~$\S$\ref{ss:q-analogue-our}
to obtain new inequalities for the \defng{$q$-order polynomial}.  Notably, we
obtain the following \defn{$q$-log-concavity}:

\smallskip

\begin{thm}[{\rm = Corollary~\ref{c:q-log-concavity}}{}]
	Let \ts $P=(X,\prec)$ \ts be a poset with \ts $|X|=n$ \ts elements.
Define
$$\Omega_q(P,t) \, : = \, \sum_{g} \ q^{|g(X)|-n}
$$
where the summation is over all \ts \defn{order preserving maps} \ts $g: X\to \{1,\ldots,t\}$,
i.e.\ maps which satisfy \. $g(x)\le g(y)$ \. for all \. $x \prec y$.
Then, for every integer \ts $t \geq 2$, we have:
\[ \Omega_q(P,t)^2  \  \geqslant_q \ \Omega_q(P,t+1) \. \cdot \. \Omega_q(P,t-1),\]
where the inequality holds coefficient-wise as a polynomial in~$q$.
\end{thm}

\smallskip

\subsection{Generalized Sidorenko inequality} \label{ss:intro-sid}
Below we give an equivalent but somewhat nonstandard reformulation
of the \defng{Sidorenko inequality} that makes it amenable for generalization.
A more traditional version is given in Section~\ref{s:sid}.

\smallskip

A \emph{chain} in a poset \. $P=(X,\prec)$ \. is a subset \. $\{x_1,\ldots,x_\ell\} \subseteq X$, such
that \. $x_1 \prec x_2 \prec \ldots \prec x_\ell\ts.$ Denote by \ts $\cC(P)$ \ts the
set of chains in~$P$.

\smallskip

\begin{thm}[{\rm Sidorenko~\cite{Sid}}]\label{t:sid}
Let \. $P=(X,\prec)$ \. and \. $Q=(X,\prec')$ \. be two posets on the same set
with \. $|X|=n$ \. elements.  Suppose
\begin{equation}\label{eq:sid-chain-condition}
\bigl|C \cap C'\bigr| \. \le \. 1 \quad \text{for all} \ \ C\in \cC(P), \ C' \in \cC(Q).
\end{equation}
Then:
\begin{equation}\label{eq:sid}
e(P) \. e(Q)\, \ge \, n!
\end{equation}
\end{thm}

\smallskip
Natural examples of posets \ts $(P,Q)$ \ts as in the theorem
are the \defn{permutation posets} \ts $\bigl(P_\si, P_{\ov \si}\bigr)$,
where \ts $P_\si=([n],\prec)$ \ts is defined as
$$
i\prec j \quad \Longleftrightarrow\quad  i<j \ \ \text{and} \ \ \si(i)<\si(j)\,, \quad \text{for all} \ \ i,j\in [n].
$$
and \.  $\ov{\si}:=\bigl(\si(n),\ldots,\si(1)\bigr)$.
In this case \ts $P_{\si}$ \ts is a \defng{$2$-dimensional poset}, and
\ts $P_{\ov \si}$ \ts is its \defng{plane dual}.

Sidorenko's original proof used combinatorial optimization and
proved also equality conditions for~\eqref{eq:sid}, see~$\S$\ref{ss:sid-proof}.
In~\cite{BBS}, the authors gave an easy reduction to a special
case of the (still open) \defng{Mahler conjecture} for \defng{convex corners}.
That special case was resolved earlier by Saint-Raymond~\cite{StR},
and was reproved and further extended in a series of papers, see \cite{AASS,BBS}
for the context and the references.

\smallskip

In this paper we give a direct injective proof of the Sidorenko inequality~\eqref{eq:sid},
which allows us to prove that the inequality is in~$\SP$ (Theorem~\ref{t:sid-SP}).
This completely resolves the open problem Morales and the last two authors in \cite{MPP}
of finding a combinatorial proof of~\eqref{eq:sid}.  Although presented differently, our
injection likely coincides with an injection of Gaetz and Gao~\cite{GG1}, see~$\S$\ref{ss:finrem-sid};
the latter was discovered independently and generalized to other Coxeter groups.

\smallskip

Our proof can also be extended to give the following generalization of Theorem~\ref{t:sid}.

\smallskip

\begin{thm}\label{t:sid-gen}
Let \. $P=(X,\prec)$ \. and \. $Q=(X,\prec')$ \. be two posets on the same set
with \. $|X|=n$ \. elements.  Suppose
$$
\bigl|C \cap C'\bigr| \. \le \. k \quad \text{for all} \ \ C\in \cC(P), \ C' \in \cC(Q).
$$
Then:
\begin{equation}\label{eq:sid-gen}
e(P) \. e(Q)\, \ge \,  \frac{n!}{k^{n-k}\. k!} \..
\end{equation}
\end{thm}

\smallskip

The proof, examples and applications of this result are given in Section~\ref{s:sid}.

\smallskip

\subsection{Generalized Stanley inequality}\label{ss:intro-Sta}
We start with the following inspiring \defn{Stanley inequality}:

\smallskip

\begin{thm}[{\rm Stanley~\cite{Sta-AF}}{}]\label{t:Sta}
Let \. $P=(X,\prec)$ \. be a poset on \. $|X|=n$ \. elements.  For
an element \. $x\in X$ \. and integer \. $1\le a \le n$, let \.
$\Ec(P,x,a)$ \. be the set of linear extensions \. $f\in \Ec(P)$ \. such
that \. $f(x)=a$.  Denote by \. $\aNr(P, x,a):=\bigl|\Ec(P, x,a)\bigr|$ \.
the number of such linear extensions.   Then:
\begin{equation}\label{eq:Sta}
\aNr(P, x,a)^2 \, \ge \, \aNr(P, x,a+1) \.\cdot \.  \aNr(P, x,a-1).
\end{equation}
\end{thm}

\smallskip

Stanley's original proof of this result is via reduction to the classical
and very deep \emph{Alexandrov--Fenchel} (AF-) \emph{inequality} in convex geometry.
While the latter has several proofs, see references in~\cite[$\S$7.1]{CP2},
none are elementary and direct even in the case of convex polytopes.
With the aim to prove \eqref{eq:Sta} by an elementary argument, the (somewhat
technical) proof in~\cite{CP} uses nothing but linear algebra.  Finding
a direct injective proof is a major open problem (see~$\S$\ref{ss:finrem-GP-injection}).

In the absence of an injective proof, the \defng{equality conditions} can become
\emph{more difficult} than the original inequality, cf.~$\S$\ref{ss:finrem-vanishing}.
This is famously the
case for  the AF-inequality and many of its consequences.  For the Stanley
inequality, the equality conditions were discovered recently by
Shenfeld and van Handel~\cite{SvH}, by a deep geometric argument.

Fortunately, part of the equality conditions called the
\defng{vanishing conditions}, are completely combinatorial.
Denote by \ts $\ell(x) := \bigl|\{y \in X\,{}:\,{}y\preccurlyeq x\}\bigr|$ \ts and \ts
$b(x) := \bigl|\{y \in X\,{}:\,{}y\succcurlyeq x\}\bigr|$ \ts the sizes of lower
and upper ideals of \ts $x\in X$, respectively.

\smallskip

\begin{thm}[{\rm Shenfeld and van Handel~\cite[Lemma~15.2]{SvH}}] \label{t:SvH-vanish}
Let \. $P=(X,\prec)$ \. be a poset on \. $|X|=n$ \. elements,
let \ts $x\in X$ \ts and \ts $1\le a \le n$.  Then \.
$\aNr(P, x,a)>0$ \. \underline{if and only if} \. $\ell(x)\le a$ \. and \. $b(x)\le n-a+1$.
\end{thm}

\smallskip

Note that by the Stanley inequality, if \. $\aN(P, x,a)=0$, then \. $\aN(P, x,a+ 1)=0$ or \.$\aN(P, x,a- 1)=0$ ,
so whenever the conditions in the theorem are not satisfied the equation~\eqref{eq:Sta} is
an equality.  We can now define the \ts \defn{generalized Stanley inequality}.

\smallskip

\begin{thm}[{\rm Stanley~\cite{Sta-AF}}{}]\label{t:Sta-gen}
Let \. $P=(X,\prec)$ \. be a poset on \. $|X|=n$ \. elements.
Fix elements \. $x,z_1,\ldots,z_k\in X$ \. and integers \. $a,c_1,\ldots,c_k\in [n]$;
we write \. $\bz =(z_1,\ldots,z_k)$ \. and \. $\bc =(c_1,\ldots,c_k)$.
Let \.
$\Ec_{\bz\ts\bct}(P,x,a)$ \. be the set of linear extensions \. $f\in \Ec(P)$ \. such
that \. $f(x)=a$ \. and \. $f(z_i)=c_i$, for all \. $1\le i \le k$.
Denote by \. $\aNr_{\bz\ts\bc}(P, x,a):=\bigl|\Ec_{\bz\ts\bc}(P,x,a)\bigr|$ \.
the number of such linear extensions. Then:
\begin{equation}\label{eq:Sta-gen}
\aNr_{\bz\ts\bc}(P, x,a)^2 \, \ge \, \aNr_{\bz\ts\bc}(P, x,a+1) \.\cdot \.  \aNr_{\bz\ts\bc}(P, x,a-1).
\end{equation}
\end{thm}

\medskip

We can now state the \defng{vanishing conditions} for the generalized Stanley inequality.
Without loss of generality, we can assume that numbers~$\bc$ are in increasing order,
in which case we can assume that elements~$\bz$ form a chain (because $\aN_{\bz\ts\bc}$ only counts linear extensions for which $z_1 \prec \cdots \prec z_k$).
%
%
Let \ts $x,y\in X$ \ts be two poset elements such that \. $x\prec y$.
Define \. $h(x,y) \ts := \ts \#\{ z\in P, \text{ s.t. } x\prec z\prec y\}$.

\smallskip

\begin{thm}\label{t:vanish}
Let \. $P=(X,\prec)$ \. be a poset on \. $|X|=n$ \. elements.
Fix elements \. $u_1 \prec\ldots \prec u_k\in X$ \. and integers \. $1\le a_1 < \ldots < a_k \le n$;
we write \. $\uu =(u_1,\ldots,u_k)$ \. and \. $\ba =(a_1,\ldots,a_k)$.
Let \.
$\Ec(P,\uu, \aa)$ \. be the set of linear extensions \. $f\in \Ec(P)$ \. such
that \. $f(u_i)=a_i$, for all \. $1\le i \le k$.
Then \. $\bigl|\Ec(P, \uu, \aa)\bigr| > 0$ \, \underline{if and only if}
\begin{equation}\label{eq:Sta-gen-vanish}
\aligned
& \ell(u_i)\. \le \. a_i \.,  \quad b(u_i)\le n-a_i+1\,, \quad \  \text{for all} \quad 1\.\le \. i \. \le \. k\ts, \ \ \, \text{and} \\
& a_j \. - \. a_i \, > \, h(u_i,u_j) \quad \  \text{for all} \quad 1\.\le \. i \. < \. j \. \le \. k\ts.
\endaligned
\end{equation}
\end{thm}

\smallskip

In Theorem~\ref{t:gen-Stanley-unique}, we also prove the \defng{uniqueness conditions}
for the problem, i.e.\ necessary and sufficient conditions for \.
$\bigl|\Ec(P, \uu, \aa)\bigr| = 1$.  We postpone the statement until~$\S$\ref{ss:restricted-unique}.

\smallskip

\subsection{Complexity implications}\label{ss:intro-CS}
We assume the reader is familiar with basic \defna{Computational Complexity},
and refer to standard textbooks \cite{AB,MM,Pap} for definitions and notation.
Here we follow the approach to inequalities proposed by the second author 
\cite{Pak,Pak-OPAC}. 

\smallskip

Recall the \defng{counting complexity class} \. $\SP$ \. of functions
which count the number of objects whose membership is decided in
polynomial time.  Let \. $\GapP=\SP-\SP$ \. be the closure of~$\SP$
under subtraction, see e.g.~\cite{For}.  Finally,
let \ts $\GapP_{\ge 0} := \GapP \cap \{u\ge 0\}$.

Clearly, \. $\SP\subseteq \GapP_{\ge 0}$\., but it remains open whether this
inclusion is proper.  For example, the \emph{Kronecker coefficients} \ts
$g(\la,\mu,\nu) \in \GapP_{\ge 0}$. It is not known whether $g(\cdot)\in \SP$,
and this remains a major open problem in Algebraic Combinatorics,
see e.g.~\cite{PP}.

\smallskip

As before, let \. $P = (X,\prec)$ \. be a poset on \ts $n$ \ts elements.
Clearly, the function \ts $e: P\to e(P)$ \ts is in $\SP$, and is famously
\ts $\SP$-complete~\cite{BW}. In fact, the function \ts $e(\cdot)$ \ts is \ts $\SP$-complete even
when restricted to permutation posets  \ts $P_\si$, $\si \in S_n$ and posets
of height two, see~\cite{DP}.  Define
\begin{equation}\label{eq:HP-function}
\xi(P) \, := \, e(P) \.\cdot\. \prod_{x \in X} b(x) \, -  \,  \. n!
\end{equation}
Observe that \. $\xi\in \GapP_{\ge 0}$ \. by the definition and
the  Bj\"orner--Wachs inequality~\eqref{eq:HP}.  In fact, the original 
injective proof of~\eqref{eq:HP} easily implies the following
effective version of the inequality:

\begin{thm} \label{t:HP-SP}
The function \. $\xi: P \to \nn$ \. defined by~\eqref{eq:HP-function} \ts is in~$\ts\SP$.
\end{thm}

\smallskip

Similarly, define \. $\ze: \ts P\times \nn \ts \to \nn$
\begin{equation}\label{eq:OP-function}
\ze(P,t) \, := \, \Omega(P,t) \. n! \, -  \,  e(P) \. t^n\..
\end{equation}
We can now give an effective version of~\eqref{eq:OP-gen}:

\smallskip

\begin{thm} \label{t:OP-SP}
The function \. $\ze: P \to \nn$ \. defined by~\eqref{eq:OP-function} \ts is in~$\ts\SP$.
\end{thm}

For every $\si\in S_n$, let \ts $\eta: S_n \to \zz$ \ts be defined
as follows:
\begin{equation}\label{eq:Sid-function}
\eta(\si)\, := \, e(P_\si) \. e\bigl(P_{\ov{\si}}\bigr)  \, -  \,  \. n!
\end{equation}
Observe that \. $\eta \in \GapP_{\ge 0}$ \. by the definition and the
Sidorenko inequality~\eqref{eq:sid}. In fact, our injective
proof of~\eqref{eq:sid} can be used to obtain the following result:

\smallskip

\begin{thm} \label{t:sid-SP}
The function \. $\eta: S_n \to \nn$ \. defined by~\eqref{eq:Sid-function} \ts is in~$\ts\SP$.
\end{thm}

\smallskip

For the vanishing conditions of the generalized Stanley inequality, the implications
are completely straightforward:

\smallskip

\begin{cor}  \label{cor:Sta-gen-vanish-poly}
In the conditions of Theorem~\ref{t:vanish}, deciding
whether \. $\bigl|\Ec(P,\uu,\aa)\bigr| > 0$ \.\ts is in $\ts\poly$.
Moreover, when \. $\bigl|\Ec(P,\uu,\aa)\bigr| > 0$, a linear
extension \. $f\in \Ec(P,\uu, \aa)$ \. can be found in polynomial time.
\end{cor}

\smallskip

We conclude with a corollary of Theorem~\ref{t:gen-Stanley-unique}.

\smallskip

\begin{cor}  \label{cor:Sta-gen-uniqueness-poly}
In the conditions of Theorem~\ref{t:vanish}, deciding
whether \. $\bigl|\Ec(P,\uu, \aa)\bigr| = 1$ \.\ts is in $\ts\poly$.
\end{cor}

\smallskip

\subsection{Structure of the paper}
The paper is written in a straightforward manner, as we devote different sections to
proofs of different results.  These proofs are completely independent and
largely self-contained.  We are hoping they will appeal to a diverse
readership.

\smallskip

We start with a short Section~\ref{s:notation}, where give some basic
notation used throughout the paper. 
In Section~\ref{s:injective}, we recall the original direct injective proof of the
Bj\"orner--Wachs inequality~\eqref{eq:HP} via direct injection 
(cf.\ Theorem~\ref{t:HP-SP}).  Here we introduce promotions of linear extensions,
a tool which will also be used later in the paper (Sections~\ref{s:restricted} and~\ref{s:sid}).

In a lengthy Section~\ref{s:OP-induction}, we use Shepp's lattice
to prove Theorem~\ref{thm:log-concave-intro} and other inequalities
for the order polynomial.  The second half of this section is
motivated by connection and applications to the \defng{Kahn--Saks Conjecture}
(Conjecture~\ref{conj:KS-mon}) and the \defng{Graham Conjecture}
(Theorem~\ref{conj:Graham}), as we prove special cases of both of them.

In the next Section~\ref{s:q-analogue}, we present
an elegant proof by Reiner of the Bj\"orner--Wachs inequality.  We then
prove a $q$-analogue of \defng{Shepp's inequality} for the $q$-analogue
of the order polynomial, by using the remarkable \defng{$q$-FKG inequality}
by Bj\"orner.  We continue with the general lower bound on the order
polynomial (Section~\ref{s:OP-inj}), and prove Theorem~\ref{thm:order_le}
by a direct injection.

In Section~\ref{s:restricted}, we prove the vanishing conditions
(Theorem~\ref{t:vanish}) and uniqueness conditions
(Theorem~\ref{t:gen-Stanley-unique}), from which
Corollaries~\ref{cor:Sta-gen-vanish-poly}
and~\ref{cor:Sta-gen-uniqueness-poly} easily follow.
Our proof is based on an algebraic approach of Coxeter group
action on linear extensions,
see~$\S$\ref{ss:finrem-algebraic} for some history of the subject.

In Section~\ref{s:sid}, we give an injective proof of the
Sidorenko inequality, and prove its extension Theorem~\ref{t:sid-gen}.
We then derive Theorem~\ref{t:sid-SP} which is surprisingly
nontrivial given the many other proofs of the inequality
(see~$\S$\ref{ss:finrem-sid}).  We conclude with Section~\ref{s:finrem}
containing lengthy historical remarks and open problems.

\medskip

\section{Basic definitions and notation}  \label{s:notation}
In a poset \ts $P=(X,\prec)$, elements \ts $x,y\in X$ \ts are called
\defn{parallel} or \defn{incomparable} if \ts $x\not\prec y$ \ts
and \ts $y \not \prec x$.  We write \. $x\parallel y$ \. in this case.
Element \ts $x\in X$ \ts is said to \defn{cover}
\ts $y\in X$, if \ts $y\prec x$ \ts and there are no elements \ts $z\in X$ \ts
such that \. $y\prec z \prec x$.

A \defn{chain} is a subset \ts $C\ssu X$ \ts of pairwise comparable elements.
The \defn{height} of poset \ts $P=(X,\prec)$ \ts is the maximum size of a chain.
An \defn{antichain} is a subset \ts $A\ssu X$ \ts of pairwise incomparable elements.
The \defn{width} of poset  \ts $P=(X,\prec)$ \ts is the size of the maximal antichain.

A \defn{dual poset} \ts is a poset \ts $P^\ast=(X,\prec^\ast)$, where
\ts $x\prec^\ast y$ \ts if and only if \ts $y \prec x$.

A \defn{disjoint sum} \ts $P+Q$ \ts of posets \ts $P=(X,\prec)$ \ts
and \ts $Q=(Y,\prec')$ \. is a poset on \ts $(X\cup Y,\prec^\di)$,
where the relation $\prec^\di$ coincides with $\prec$ and $\prec'$ on
$X$~and~$Y$, and \. $x\.\|\. y$ \. for all \ts $x\in X$, $y\in Y$.

A \defn{linear sum} \ts $P\oplus Q$ \ts of posets \ts $P=(X,\prec)$ \ts
and \ts $Q=(Y,\prec')$ \. is a poset on \ts $(X\cup Y,\prec^\di)$,
where the relation $\prec^\di$ coincides with $\prec$ and $\prec'$ on
$X$~and~$Y$, and \. $x\prec^\di y$ \. for all \ts $x\in X$, $y\in Y$.

A \defn{product} \ts $P\times Q$ \ts of posets \ts $P=(X,\prec)$ \ts
and \ts $Q=(Y,\prec^\ast)$ \. is a poset on \ts $(X\times Y,\prec^\di)$,
where the relation \. $(x,y) \preccurlyeq^\di (x',y')$  \. if and only if
\. $x \preccurlyeq x'$ \. and \. $y \preccurlyeq^\ast y'$,
for all \ts $x,x'\in X$ \ts and \ts $y,y'\in Y$.

Posets constructed from one-element posets by recursively taking
disjoint and linear sums are called \defn{series-parallel}.
Both \defn{$n$-chain} \ts $C_n$ \ts and \defn{$n$-antichain} \ts $A_n$
 \ts are examples of series-parallel posets.

For a subset \ts $Y\ssu X$, a \ts \defn{restriction} \ts of the
poset \ts $=(X,\prec)$ \ts to \ts $X\sm Y$ \ts is a subposet
\ts $(X\sm Y,\prec)$ \ts of~$P$, which we denote
by \ts $P\sm Y$ \ts and \ts $P|_{X\sm Y}$.

For a poset \ts $P=(X,\prec)$, a function \ts $f: X\to \rr$ \ts
is called \defn{$\prec$-increasing} \ts if \ts $f(x)\le f(y)$ \ts
for all \ts $x\preccurlyeq y\ts;$ such functions are also called
\defn{weakly order-preserving} in a different context.
The \ts \defn{$\prec$-decreasing} \ts functions
are defined analogously.

Throughout the paper we use \, $\nn=\{0,1,2,\ldots\}$, \.
$\pp=\nn_{\ge 1} =\{1,2,\ldots\}$ \. and \. $[n]=\{1,\ldots,n\}$.

\medskip

\section{Injective proof of the Bj\"orner--Wachs inequality}\label{s:injective}
%
%
In this short section we recap the original proof by Bj\"orner and Wachs.  We do
this both as a warmup and as a way to introduce some definitions and ideas
that will prove useful throughout the paper.  As a quick application, we
obtain the proof of Theorem~\ref{t:HP-SP}.  The reader well familiar
with~\cite{BW89} can skip this section.

\smallskip

Denote by \ts $\cS=\cS(P)$ \ts the set of all bijections
\. $\si: X\to [n]$, so that \. $\Ec(P) \subseteq \cS(P)$. Denote by \ts $\cB=\cB(P)$ \ts
the set of maps \. $g: X\to X$ \. such that \. $g(x)\succcurlyeq x$ \. for all \. $x\in X$.
The inequality~\eqref{eq:HP} can then be written as:
$$(\ast) \qquad
\bigl|\cS(P)\bigr| \, \le \, \bigl|\Ec(P)\bigr| \. \cdot \.  \bigl|\cB(P)\bigr| \..
$$
We prove \ts $(\ast)$ \ts by a direct injection \. $\Phi: \cS \to \Ec \times \cB$ \.
defined as follows.

We say that a bijection \ts $f: X\to [n]$ \ts is \defn{sorted} on a subset \ts $Y \subseteq X$,
if \. $f(x)< f(y)$ \. for all \. $x, \ts y\in Y$ \. such that \. $x\prec y$.
Fix a linear extension \ts $\al\in \Ec(P)$ and label the elements of $X$ naturally according to $\al$, so that $x_i = \al^{-1}(i)$.  For every \. $\si\in \cS$,
proceed with the following \defng{sorting algorithm} for the elements \.
$x_n,\ldots,x_1$ \. in this order.  At the $k$-th step,
take the element \. 
$x=x_{n-k+1}$ \. and let \. $f(x):=\si(x)$.
If \ts $\si(x)$ is the smallest of \ts $\{f(y), y\in B(x)\}$, do nothing.
Otherwise, start the \defn{demotion} of \ts $x$ \ts by  swapping
its value $f(x)$ with the smallest $f(x')$, where \. $x'\in B(x)$.
Repeat this with $x'$, etc., until all elements in \ts $B(x)$ \ts are sorted.
Let \ts $g(x):=y$ \ts be the element in $B(x)$ where $x$ is demoted to (i.e. $y$ is the largest element in $B(x)$ affected by the demotion).

At end of the sorting algorithm, we obtain a bijection
\ts $f: X\to [n]$ \ts that is sorted on the whole~$X$, i.e. $f \in \Ec(P)$.
We also obtain a map \ts $g\in \cB(P)$.  Define  \. $\Phi(\si):=(f,g)$.

\smallskip

\begin{prop}[{\cite{BW89}}]
\label{p:HP-injection}
The map \. $\Phi: \ts \cS(P) \. \to \. \Ec(P)\ts\times \ts\cB(P)$ \. defined above is an injection.
\end{prop}

\begin{proof}
Define the \emph{inverse construction} as follows. Proceed through the reverse order of the
elements \. 
$x_1,\ldots,x_n$.  At \ts $(n-k)$-th step,
take the element \. $x=x_{n-k}$
\. and let \. $y=g(x)$. At this step, the
bijection \. $f: X\to [n]$ \. is sorted on \. $B(x)$.

Start the \defn{promotion} of $y$ by swapping $f(y)$ with the maximal $f(y')$,
over all \. $x\preccurlyeq y' \prec y$, until eventually $f(y)$ is promoted to
the element~$x$.  Denote by \. $\si\in \cS(P)$ \. the result of this iterated promotion and
define a map \. $\Psi(f,g):=\si$.

Now observe that for all \. $\si\in \cS(P)$ \. we have \. $\Psi(\Phi(\si))=\si$, since
the map \. $\Psi$ \. retraces each step of $\Phi$ by the properties of promotion and demotion.
This implies that $\Phi$ is an injection and completes the proof of the claim.
\end{proof}

\smallskip

%

\begin{proof}[Proof of Theorem~\ref{t:HP-SP}]
Let \. $\cH(P)\subseteq \Ec(P) \times \cB(P)$ \. be the set of pairs \.
$(f,g)\in \Ec(P) \times \cB(P)$, such that \. $\Phi(\Psi(f,g))\ne (f,g)$.
By definition, \. $\bigl|\cH(P)\bigr| = \xi(P)$.
Since both $\Phi$ and $\Psi$ are computable in polynomial time, then
so is the membership in~$\cH(P)$.  This proves the result.
\end{proof}

\smallskip

A series-parallel poset \ts $P=(X,\prec)$ \ts is called an \defn{ordered forest}
if it is a disjoint union of rooted trees, where each tree is rooted it its unique minimal element.

\smallskip

\begin{prop}[{\cite{BW89}}]\label{p:HP-forest}
The Bj\"orner--Wachs inequality~\eqref{eq:HP} is an equality
\ts \underline{if and only if} \ts $P$ \ts is an ordered forest.
\end{prop}

\begin{proof}
As mentioned in the introduction, for the ``if'' direction, the equality
can be easily proved by induction, see e.g.~\cite{Bea,SY}.
For the ``only if'' direction, let \. $x,y,z\in X$ \. be a poset elements such that \.
$x\prec z$, \. $y\prec z$, \. $x \.\|\. y$, and such that
both elements \ts $x,y$ \ts are covered by~$z$.
We claim that \eqref{eq:HP} is a strict inequality in this case.

In the notation above,
choose \. $g\in \cB(P)$ \. such that \. $g(x)=g(y)=z$ \. and \. $g(s)=s$  \. for all $s \in X \sm \{x,y\}$.  It is easy
to see that there exists \. $f\in \Ec(P)$, such that \. $f(x)= k-1$, \.
$f(y)=k$ \. and \. $f(z)=k+1$, for some \. $1< k< n$. Assume that $x$ precedes $y$ in the natural labeling.
Applying $\Psi$ we see that after promoting $f(x)$ to $z$, the result is no longer a linear extension on $B(y)$ and thus \. $\Psi \. $ is not defined there. Thus \.$\Phi$\. is not a bijection, which proves the claim.

Finally, observe that $P$ is an ordered forest if and only if every poset
element covers at most one element. This proves the result.
\end{proof}

\medskip

\section{Bounding order polynomial by the FKG inequality} \label{s:OP-induction}

In this section we prove the bound on the order polynomial from
Theorem~\ref{thm:HP order poly} using an inductive approach and an application
of the FKG inequality on the Shepp's lattice.

\smallskip

\subsection{Shepp's lattice and the FKG inequality}\label{ss:Shepp lattice}
Recall that a \defnb{lattice} \. $\ll:=(L,\prec^\di)$ \. is a partially ordered set on~$L$,
such that every \ts $a,b\in L$ \ts has a unique least upper bound called \defnb{join} \ts $a \vee b$,
and a unique greatest lower bound called \defnb{meet} \ts $a \wedge b$.
A lattice is called \defnb{distributive} \ts if
\[ a \wedge (b \vee c) \ = \ (a \wedge b) \vee (a \wedge c) \quad \text{ for all } \ a, \ts b, \ts c\ts \in L\ts.
\]
A function \. $\mu: L\to \Rb_{\geq 0}$ \. is called \defnb{log-supermodular} \ts
if
\[ \mu(a) \.  \mu(b) \ \leq \  \mu(a \wedge b) \. \mu(a \vee b) \quad \text{ for all } \ a,\ts b \in L\ts.
\]

Fix a positive integer \ts $t>0$.
Let \. $P=(X,\prec)$ \. be a poset on \ts $|X|=n$ \ts elements, let \. $X=Y \sqcup Z$ \.
be a partition of $X$ into two disjoint subsets.
\defnb{Shepp's lattice} \. $\ll= \ll_{Y,Z,t} \ := \ (L,\prec^\di)$ \. is defined as
	\[
 L \ := \ \big\{ \bv = (v_x)_{x \in X}  \, :  \,  1 \. \leq \. v_x \. \leq \.  t \big\},
 \]
and let
\begin{equation*}
\bv \. \preccurlyeq^\di \. \bw  \quad \Longleftrightarrow
\quad \left\{\aligned  & \, v_y \. \leq \. w_y   & \text{for all} \ \ y \in Y \\
& \, v_z \. \geq \. w_z & \text{for all} \ \ z \in Z \endaligned
\right.
\end{equation*}
Let  \. $\mu=\mu_{Y,Z} \ts : \ts L \to \{0,1\}$ \. be a function defined as
\begin{equation*}
\mu(\bv)=1 \quad \Longleftrightarrow
\quad  v_{x} \ \leq \ v_{x'}  \quad  \text{for all} \ \  x \preccurlyeq x' \ \ \
\text{such that} \ \ \ \left[ \. \aligned & x,x' \in Y\ts, \ \, \text{or} \\
& x,x'\in Z\ts. \endaligned\right.
\end{equation*}

\smallskip

\begin{thm}[\cite{She}]
	Let \ts $P=(X,\prec)$ \ts be a finite poset, and let \ts $X=Y \sqcup Z$ \ts be a partition of
the ground set $X$ into two disjoint subsets. Then
	\. $\ll_{Y,Z,t}$ \. is a distributive lattice, and \. $\mu_{Y,Z}$ \. is a log-supermodular function.
\end{thm}

\smallskip

This beautiful result is relatively little known;
we include a short proof for completeness.

\smallskip

\begin{proof}
	It follows from the definition of \. $\prec^\di$\ts,  that for all \. $\bv, \bw \in L$, we have:
	 \begin{equation}\label{eq:She1}
	 \begin{split}
	 	& (\bv \wedge \bw)_{y} \ = \  \min\{v_y, w_y\}, \qquad  (\bv \wedge \bw)_{z} \ = \  \max\{v_z, w_z\},\\
	 	& (\bv \vee \bw)_{y} \ = \  \max\{v_y, w_y\}, \qquad  (\bv \vee \bw)_{z} \ = \  \min\{v_z, w_z\},
	 \end{split}
	 \end{equation}
 where \ts $y \in Y$ \ts and \ts $z \in Z$.
 Now note that, for all real numbers \. $\alpha,\beta, \gamma \in \Rb$\.,
 \begin{equation}\label{eq:She2}
 	\min\{\alpha, \max\{\beta,\gamma\} \} \ = \ \max \{ \min \{\alpha,\beta\}, \. \min \{\alpha,\gamma\} \}.
 \end{equation}
It then follows from \eqref{eq:She1} and \eqref{eq:She2} that \ts $\ll$ \ts is a distributive lattice.

To show that $\mu$ is a log-supermodular function,
it suffices to verify the cases when \. $\mu(\bv)= \mu(\bw)=1$\ts.
Let \. $y,y' \in Y$ \. be such that $y \prec y'$.
Note that \. $v_y \leq v_{y'}$ \. and \. $w_y \leq w_{y'}$.
Then we have:
\[
(\bv \wedge \bw)_y \ = \ \min\{v_y, w_y\}  \ \leq \  \min\{v_{y'}, w_{y'}\} \ = \  (\bv \wedge \bw)_{y'}\..
\]
Similarly, we have \. $(\bv \wedge \bw)_{z} \. \leq \. (\bv \wedge \bw)_{z'}$ \. for all  \. $z \prec z'$, \. $z, z' \in Z$.
Therefore, \. $\mu(\bv \wedge \bw)=1$.
Analogously, we also have \. $\mu(\bv \vee \bw)=1$\., and the proof is complete.
\end{proof}

\smallskip

\begin{rem}
Shepp's lattice \ts $\ll$ \ts used in this section should not be confused with another lattice
defined in~\cite{She-XYZ} by Shepp.  Both lattices share the same ground set but have different
partial orders, and the partial order of the lattice in \cite{She-XYZ} was specifically chosen
to prove the \defng{$XYZ$ \ts inequality}.
\end{rem}

\smallskip

Now recall the classical \defn{FKG inequality}, see e.g.~\cite[$\S$6.2]{ASE}.

\smallskip

\begin{thm}[{\rm {\em{\defn{FKG inequality}}}, \cite{FKG}}]\label{thm:FKG}
	Let \. $\ll=(L,\prec)$ \. be a finite distributive lattice, and let
\. $\mu: L \to \Rb_{\geq 0}$ \. be a log-supermodular function.
Then, for every \ts $\prec^\di$-decreasing functions \.
$g,h: \ts L \to \Rb_{\geq 0}$\ts, we have:
\begin{equation}\label{eq:FKG-ineq}
E(\textbf{\em 1}) \. E(gh) \ \geq  \  E(g) \. E(h),
\end{equation}
	where \.
	\[ E(g) \, = \, E_\mu(g) \ := \  \sum_{x \in L} \, g(x) \. \mu(x)\., \]
and function \. $\textbf{\em 1}:L \to \Rb$ \. is given by \. $\textbf{\em 1}(x)=1$ \. for all \ts $x \in L$.

Furthermore, the inequality~\eqref{eq:FKG-ineq} also holds when both \ts $g,h$ \ts are $\prec^\di$-increasing.
On the other hand, when \ts $g$ is $\prec^\di$-decreasing and $h$ is $\prec^\di$-increasing,
 the inequality~\eqref{eq:FKG-ineq} is reversed.
\end{thm}

\smallskip

Below we apply the FKG inequality to Shepp's lattice to prove several
inequalities for the order polynomial.

\medskip

\subsection{Correlation inequalities} \label{ss:OP-induction-prelim}
%
Let \. $P=(X,\prec)$ \. be a poset on $n$ elements.  As above, denote by \ts $\cS=\cS(P)$ \ts the
set of bijections \ts $f: X\to [n]$.
%
By an abuse of notation, for every (not necessarily distinct) elements \. $u,  v\in X$,
we write
\[
\{\ts u \precc v  \ts \}  \qquad \text{as a shorthand for the collection} \qquad
\{\ts f \in \cS \. : \. f(u) \le  f(v) \ts \}\..
\]
One can write the set of linear extensions \ts $\Ec(P)$ \ts as the intersection of
collections \. $\{u\precc v\}$, for all pairs \ts $u\prec v$ \ts in~$P$.  Conversely,
every such intersection is a set of linear extensions of the corresponding poset.
The language of \defn{collections} is technically useful for our purposes.

Let \. $X= Y \. \sqcup \. Z$ \. be a \defn{partition} of~$X$ into two disjoint subsets.
%
A collection \ts $C$ \ts is called \ts \defn{$Y$-minimizing} \ts w.r.t.\
the partition \ts $Y\sqcup Z$ \ts
if $C$ is an intersection of collections of the form \. $\{y\precc z\}$,
for \. $y\in Y$ \. and \. $z\in Z$.
%
Similarly,
a collection \ts $C$ \ts is called \defn{$Y$-maximizing} w.r.t.\  the partition \ts $Y\sqcup Z$ \ts
if $C$ is an intersection of collections of the form \. $\{z\precc y\}$,
for \. $y\in Y$ \. and \. $z\in Z$.
By a slight abuse of notation, we write \ts $\Omega(C,t)$ \ts
to denote the order polynomial of a poset given by the collection~$C$.

For the rest of this section, let \ts $A$ \ts be the collection given by
\begin{equation*}\label{eq:colA}
	A \ := \  \bigcap_{y\prec y', \ y, y'\in Y} \. \{ y \precc y' \}\ts \ \cap \ \bigcap_{z\prec z', \ z, z'\in Z} \. \{ z \precc z' \}\ts,
\end{equation*}
the collection of events involving only elements of $Y$ or only elements of $Z$.
\smallskip

\begin{lemma}[{\cite[Eq.~(2.12)]{She}}]\label{t:FKG order polynomial}

In the notation above,  let \ts $C,C'$ \ts
be $Y$-minimizing collections w.r.t.\ partition \ $X=Y\sqcup Z$. Then, for every integer \ts $t>0$, we have:
\[
\Omega \big(C  \cap   C' \cap A,\ts t\big)  \. \cdot \.  \Omega\big( A,\ts t\big)
\quad \geq \quad \Omega \big(C \cap A,\ts t\big)  \. \cdot \.  \Omega\big( C' \cap A,\ts t\big)\..	
\]
If $C$ is $Y$-minimizing and $C'$ is $Y$-maximizing, then the above inequality is reversed.
\end{lemma}

\smallskip

In the probabilistic language, it says that the order polynomial
satisfies \defna{positive correlation} for intersections of $Y$-minimizing
collections (viewed as events).  We should note that in~\cite{She} this
result was not singled out and appears as an equation in the middle of
the proof of the main result.  We again include the proof for completeness.

\smallskip

\begin{proof}[Proof of Lemma~\ref{t:FKG order polynomial}]
Let \ts $\ll=(L,\prec^\di)$ \ts be Shepp's lattice defined in~$\S$\ref{ss:Shepp lattice}.
Let \. $g,h : L \to \{0,1\}$ \. be given by
	\[  g(\bv)  \ := \
\begin{cases}
	\. 1 & \text{if} \ \  v_y  \ts \leq \ts v_z \ \ \, \forall \. \{ y \precc z\} \in C\\
	\. 0 & \text{otherwise}
\end{cases} \quad \text{and}
\quad  h(\bv)  \ := \
\begin{cases}
	\. 1 & \text{if} \ \  v_y  \ts \leq \ts v_z  \ \ \, \forall \. \{ y \precc z\} \in C'\\
	\. 0 & \text{otherwise.}
\end{cases}
\]
Let us prove that \ts $g$ \ts is $\prec^\di$-decreasing.
It suffices to show that
\begin{equation*}
g(\bw)\ts =\ts 1 \quad \text{and} \quad \bv \prec^\di \bw \quad \Longrightarrow \quad  g(\bv)\ts = \ts 1.
\end{equation*}
Note that for every \ts $y \in Y$ \ts and \ts $z \in Z$ \ts such that \ts $y \preccurlyeq z$, we have:
 \[  v_y \ \leq \  w_y \ \leq \ w_z \ \leq \ v_z, \]
where the first and the third inequality is because \ts $\bv \prec^\di \bw$\ts, and the second inequality is because $g(\bw)=1$.
This implies that $g(\bv)=1$, and thus $g$ is a $\prec^\di$-decreasing function.
By the same reasoning, we also have that function \ts $h$ \ts is $\prec^\di$-decreasing.

Finally, note that
\begin{align*}
	\Omega(C \cap C' \cap A, t) \. = \. E(gh), \ \  \Omega( A, t) \. = \. E(\textbf{1}),
	\ \  \Omega(C  \cap A, t) \. = \. E(g),
	\ \  \Omega(C' \cap A, t) \. = \. E(h),
\end{align*}
The lemma now follows from the FKG inequality (Theorem~\ref{thm:FKG}).
\end{proof}

\smallskip

We can now apply this result in the more traditional notation of order polynomials
of posets.

\smallskip

\begin{lemma}\label{l:main order polynomial}
	Let \. $P=(X,\prec)$ \. be a poset, and let \. $x,y\in X$ \. be minimal elements.
	Then, for every integer $t>0$, we have:
\[ {\Omega\big(P,\ts t\big)} \. \cdot \. {\Omega\big(P \sm \{x,y\},\ts t\big)}  \ \geq \  {\Omega\big(P \sm x, \ts t\big)}  \. \cdot \.  {\Omega\big(P\sm y,\ts t\big)},
\]
where by \ts $P\sm x$, \ts $P\sm y$ \ts and \ts $P\sm \{x,y\}$ \ts denote the
subposets of~$\ts P$ \ts restricted to \ts $X-x$, \ts \ts $X-y$ \ts and \ts $X-x-y$, respectively.
\end{lemma}

\smallskip

\begin{proof}
Note that \. $x$ \ts and \ts $y$ \. are incomparable elements.
Let \. $Y:= \{x,y \}$ \. and \. $Z:= X\sm Y$ \. be the partition of~$X$.
Consider the $Y$-minimizing collections \. $C$ \ts and \ts $C'$ \. given by
\begin{equation}\label{eq:colC}
C  \ := \ \bigcap_{z\in B(x)-x} \. \{ x \precc z \} \ \, \quad \text{and} \quad \  C'  \ := \  \bigcap_{z'\in B(y)-y} \. \{ y \precc z' \},
\end{equation}
where \ts $B(x)$ \ts and \ts $B(y)$ \ts are upper order ideals of elements~$x$ and~$y$, respectively.
Observe that
	\begin{alignat*}{2}
	 & \Omega(P,t) \ = \  \Omega(C \cap C' \cap A,t)\., \qquad  && \Omega(P\sm x,t) \ = \  \frac{1}{t} \,\. \Omega(C' \cap A,t)\.,\\
	    	& \Omega(P \sm y,t) \ = \  \frac{1}{t} \, \Omega(C \cap A,t)\., \qquad  && \Omega(P \sm \{x,y\},t) \ = \  \frac{1}{t^2} \, \Omega(A,t)\..
	\end{alignat*}
	The lemma now follows from  Lemma~\ref{t:FKG order polynomial}
	and the equations above.
\end{proof}

\smallskip

\subsection{Lower bounds}\label{ss:OP-induction-lower}
We are now ready to prove Theorem~\ref{thm:HP order poly} by induction.
The following lemma established the induction step from which the
theorem follows.

\smallskip

\begin{lemma}\label{l:OP-HP-ind}
Let \. $P=(X,\prec)$ \. be a finite poset, and let \ts $x\in X$ \ts
be a minimal element.  Assume that \ts $b(x)>1$.
Then we have:
\begin{equation}\label{eq:HP-ind-min}
\frac{\Omega(P,t)}{t+1} \ \geq \ \frac{\Omega(P\sm x,t)}{b(x)}\..
\end{equation}
\end{lemma}

\begin{proof}  Let \ts $Q=(Y,\prec)$ \ts be a finite poset.
By labeling the elements in~$Y$ with $m$ distinct integers from~$\ts [t]$,
we can write
\begin{align}\label{eq:sum_ideals}
\Omega(Q,t) \ = \ \sum_{m=1}^n \, \bigl|\cI_m(Q)\bigr| \, \binom{t}{m}\.,
\end{align}
where \. $|\cI_m(Q)|$ \. is the number of ascending chains
$$\emp \. = \. I_0 \. \subset \. I_1 \. \subset \. I_2 \. \subset \. \ldots \. \subset \. I_m \. = \.Q
$$
\. of  upper order ideals in~$P$, s.t.\ \. $I_i \sm I_{i-1} \neq \emp$ \. for all \ts $1\le i \le m$.

Let \. $n:=|X|$ \. be the number of elements in~$X$.
Denote by \. $P' = P \sm \{x\}$ \. the induced poset on \ts $X \sm x$.
Suppose that $x$ is a unique minimal element of $P$, and let \. $P'=P\sm x$.
Note that \. $b(x)=n$ \. in this case.
Summing over all possible values of~$x$, we obtain~\eqref{eq:HP-ind-min}:
\begin{align*}
\Omega(P,t) \ &=  \ \sum_{k=1}^t \. \Omega(P',k) \ =_{\eqref{eq:sum_ideals}} \
\sum_{m=1}^{n-1} \, \bigl|\cI_m(P')\bigr| \, \sum_{k=m}^t \. \binom{k}{m} \ = \ \sum_{m=1}^{n-1} \, \bigl|\cI_m(P')\bigr| \,  \binom{t+1}{m+1} \\
& \geq \ \frac{t+1}{n} \, \sum_{m=1}^{n-1} \, \bigl|\cI_m(P')\bigr| \,   \binom{t}{m} \ =_{\eqref{eq:sum_ideals}} \ \frac{t+1}{n} \ \Omega(P',t)\..
\end{align*}
Here the inequality follows from
$$\binom{t+1}{m+1} \ = \ \frac{t+1}{m+1} \. \binom{t}{m} \ \geq \ \frac{t+1}{n} \. \binom{t}{m} \quad \ \.\text{for all \ $m\leq n-1$.}
$$

Suppose now that \ts $x\in X$ \ts is not a unique minimal element.  Let \ts $y\in X$, \ts $y\ne x$ \ts
be another minimal element in~$P$.  By Lemma~\ref{l:main order polynomial}, we have:
\begin{equation}\label{eq:FKG2}
\frac{\Omega(P,t)}{\Omega(P',t)} \ \geq \ \frac{\Omega(P \sm y, t)}{\Omega(P' \sm y, t)} \,.
\end{equation}
Now proceed by induction to remove all minimal elements in~$P$ incomparable to~$x$,
until element~$x$ becomes
the unique minimal element.
Applying the inequality~\eqref{eq:FKG2} repeatedly, we obtain:
\begin{equation*}
\frac{\Omega(P,t)}{\Omega(P',t)} \ \geq \
\ldots \ \geq \ \frac{t+1}{b(x)}\,.
\end{equation*}
This proves~\eqref{eq:HP-ind-min} in full generality.
\end{proof}

\smallskip

\begin{proof}[Proof of Theorem~\ref{thm:HP order poly}]
We prove the inequality~\eqref{eq:OP-HP} by induction.
First, suppose that \ts $b(x)=1$, so \ts $x$ \ts is the maximal element in~$P$.
Then \. $\Omega(P,t) \ts = \ts t \. \Omega(P',t)$, since we can choose the value \ts
$f(x)\in [t]$ \ts independently of other values.  The inequality~\eqref{eq:OP-HP}
follows then.  For \. $b(x)>1$,  Lemma~\ref{l:OP-HP-ind} gives the step of
induction and complete the proof.
\end{proof}

\medskip

\subsection{Log-concavity}\label{ss:OP-log-concavity}
The main result of this subsection is the log-concavity of the evaluation of the order polynomial.

\smallskip

\begin{thm}\label{thm:log-concave}
	Let $P=(X,\prec)$ be a finite poset.
	Then, for every integer \ts $t \geq 2$, we have:
	\[ \Omega(P,t)^2  \  \geq \ \Omega(P,t+1) \. \cdot\.  \Omega(P,t-1).\]
\end{thm}

\smallskip

As for other poset inequalities, one can ask about equality conditions
in Theorem~\ref{thm:log-concave}.  Turns out, the log-concavity in the theorem
is always strict, see Theorem~\ref{thm:log-concave-strict} below.
The proof of both results use the same approach, but the strict
log-concavity is built on top of the non-strict version and is
a bit more involved.  Thus, we start with the easier result for clarity.

\smallskip

\begin{proof}[Proof of Theorem~\ref{thm:log-concave}]
Let \ts $Y=X$ \ts and \ts $Z=\varnothing$.
Let \ts $\ll=(L,\prec^\di)$ \ts be Shepp's lattice defined in~$\S$\ref{ss:Shepp lattice}, and let \ts $t\geq 3$.
		 Let \. $g,h : L \to \{0,1\}$ \. be two functions given by
				\begin{align*}
			g(\bv) \ &:= \
			\begin{cases}
				\. 1 & \text{if} \ \ v_x \geq 2 \ \ \text{ for all \ $x \in X$}\\
				\. 0 & \text{otherwise},
			\end{cases} \quad
			h(\bv) \ := \
			\begin{cases}
				\. 1 & \text{if} \ \ v_x \leq t-1 \ \ \text{ for all \ $x \in X$}\\
				\. 0 & \text{otherwise}.  	
			\end{cases}
		\end{align*}
To prove that $g$ is $\prec^\di$-increasing, it suffices to show that
\begin{equation*}
g(\bv)\ts =\ts 1 \quad \text{and} \quad \bv \prec^\di \bw \quad \Longrightarrow \quad  g(\bw)\ts = \ts 1.
\end{equation*}
Note that, for every \ts $x \in Y=X$, we have:
\[
w_x \, \geq \,  v_x \, \geq \, 2,
\]
where the first  inequality is because \ts $\bv \prec^\di \bw$\ts, and the second inequality is because $g(\bv)=1$.
This implies that $g(\bw)=1$, and thus $g$ is a $\prec^\di$-increasing function.
By an analogous reasoning we also have that $h$ is $\prec^\di$-decreasing.	
	
		Now note that
		\begin{equation}\label{eq:expectation-as-OP}
		\begin{split}
			E(gh) \ &= \  \big|\big\{ \bv \in L \, : \,  2 \. \leq \. v_x \. \leq \. t-1  \, \ \text{ for all } \,\  x \in X\big\} \big| \ = \  \Omega(P,t-2),\\
			E(g)  \ &= \   \big|\big\{ \bv \in L \. : \,  2 \. \leq \. v_x \, \ \text{ for all } \, \ x \in X\big\} \big| \ = \   \Omega(P,t-1),\\
			E(h) \ &= \  \big|\big\{ \bv \in L \, : \,  v_x \. \leq \. t-1  \,\ \text{ for all } \, \ x \in X\big\} \big| \ = \  \Omega(P,t-1),\\
			E(\textbf{1})  \ &= \  |L| \ = \ \Omega(P,t).
		\end{split}
		\end{equation}
		It then follows from the FKG inequality~(Theorem~\ref{thm:FKG}), that
		\[ \Omega(P,t-2) \.\cdot\. \Omega(P,t) \ \leq  \ \Omega(P,t-1) \.\cdot\. \Omega(P,t-1), \]
		and the theorem now follows by substituting \. $t\to t+1$.
\end{proof}

\medskip

\subsection{Strict log-concavity}
We can now prove that the inequality in Theorem~\ref{thm:log-concave} is always strict,
by applying the FKG inequality in a more careful manner.  This theorem will also
proved useful in~$\S$\ref{ss:OP-Graham} to establish the strict asymptotic version of
Graham's Conjecture~\ref{conj:Graham}.

\smallskip

\begin{thm}\label{thm:log-concave-strict}
	Let \ts $P=(X,\prec)$ \ts be a poset with \ts $|X|=n$ \ts elements.
	Then, for every integer \ts $t \geq 2$, we have:
	\begin{equation}\label{eq:log-concave-strict}
	 \Omega(P,t)^2  \  \geq \ \bigg(1 \. + \. \frac{1}{(t+1)^{n+1}} \bigg) \    \Omega(P,t+1) \   \Omega(P,t-1).
	\end{equation}
\end{thm}

	
\smallskip

Let \ts $Y=X$ \ts and \ts $Z=\varnothing$.
Let \ts $\ll=(L,\prec^\di)$ \ts be Shepp's lattice, let \ts $\mu = \mu_{Y,Z} \ts : \ts L \to \{0,1\}$ \ts
be the log-supermodular function defined in~$\S$\ref{ss:Shepp lattice}, and let \ts $t\geq 3$.
Without loss of generality, assume that \ts $X=[n]$ \ts and that this is a \defn{natural labeling}
of~$X$, i.e.\  \. $i<j$ \. for all \. $i \prec j$.
	
	For all \ts $1\le i\le n$,
	let \. $g_i,h_i : L \to \{0,1\}$ \. be two functions given by
	\begin{align*}
		g_i(\bv) \ &:= \
		\begin{cases}
			\. 1 & \text{if} \ \ v_i \geq 2 \\
			\. 0 & \text{otherwise}
		\end{cases} \qquad \text{and} \qquad
		h_i(\bv) \ := \
		\begin{cases}
			\. 1 & \text{if} \ \ v_i \leq t-1 \\
			\. 0 & \text{otherwise}.	
		\end{cases}
	\end{align*}
In notation of the proof of Theorem~\ref{thm:log-concave},
we have \. $g=g_1\cdots g_n$ \. and  \. $h=h_1\cdots h_n$.

It follows from the same argument above, that \ts $g_i$ \ts are $\prec^\di$-increasing,
while \ts $h_i$ \ts are $\prec^\di$-increasing, for all \ts $1\le i\le n$.
We now show that $g_i$ and $h_i$ are log-supermodular functions.
Indeed, note that
\[   v_i \geq 2 \ \text{ and } \ w_i \geq 2 \qquad \Longleftrightarrow \qquad \max\{v_i, w_i \} \geq 2  \ \text{ and } \ \min\{v_i, w_i \} \geq 2.  \]
This implies that 	\. $g_i(\bv) \. g_i(\bw) \ = \ g_i(\bv \wedge \bw) \.
g_i(\bv \vee \bw)$, as desired.
The same argument implies that $h_i$ is also a log-modular function.

\smallskip

\begin{lemma}
	\begin{equation}\label{eq:remove-g}
		 \frac{E_\mu(g_1\ts \cdots \ts g_n \cdot h)}{E_\mu(g_1\ts\cdots \ts g_n)}  \ \leq \
	\frac{E_\mu(g_n \ts h )}{E_\mu(g_n)}.
	\end{equation}
\end{lemma}

\begin{proof}
Let \ts $\mu_i: L \to \Rb$ \ts be given by \. $\mu_i :=  (g_i  \ts \cdots \ts g_n) \ts \mu$,
for all \. $1\le i \le n$, and let \ts $\mu_{n+1}:=\mu$.
Note that function \ts $\mu_i$ \ts is a log-supermodular since it is a product of
log-modular and log-supermodular functions.
Therefore, for all \ts $2\le i \le n$, we have:
	\begin{equation}\label{eq:remove-g-induction}
		 \frac{E_{\mu_{i}}(g_{i-1} \ts h)}{E_{\mu_{i}}(g_{i-1} \ts )} \ \leq \ \frac{E_{\mu_{i}}( h)}{E_{\mu_{i}}(\textbf{1})} \ = \ \frac{E_{\mu_{i+1}}(g_i \ts h)}{E_{\mu_{i+1}}(g_i)}\,,
	\end{equation}
where the inequality is by  the FKG inequality (Theorem~\ref{thm:FKG}) applied to the  $\prec^\di$-increasing function~$g_{i-1}\ts$,
to the $\prec^\di$-decreasing function~$\ts h$, and to the log-supermodular function~$\ts\mu_i\ts$.
We conclude:
\[ 		 \frac{E_\mu(g_1\ts\cdots\ts g_n \cdot h)}{E_\mu(g_1\ts\cdots\ts g_n)}  \ = \  \frac{E_{\mu_2}(g_1 \ts h)}{E_{\mu_2}(g_1 )}   \leq \
\frac{E_{\mu_{n+1}}(g_n \ts h )}{E_{\mu_{n+1}}(g_n)}  \ = \ \frac{E_{\mu}(g_n \ts h )}{E_{\mu}(g_n)}\,,  \]
where the inequality is by consecutive applications of \eqref{eq:remove-g-induction}.
\end{proof}

\smallskip

Now, let \. $\eta_i:= (h_i\ts \cdots \ts h_n) \ts  \mu$ \. for all \. $1\le i\le n$, and let \. $\eta_{n+1}:=\mu$.
	Again, note that \ts $\eta_i$ \ts is a log-supermodular function.
	Observe  that \. $E_{\eta_i}[g_n] \. = \. E_{\eta_{i+1}}[g_n \ts h_i]$ \. for all \. $2\le i \le n$.
	This implies that
\begin{equation}\label{eq:remove-h}
	\frac{E_{\mu}(g_n \ts h )}{E_{\mu}(g_n)} \ = \  \prod_{i=2}^{n+1} \. \frac{E_{\eta_i}(g_n \ts h_{i-1})}{E_{\eta_i}(g_n)}\..
\end{equation}

We apply two different inequalities to the RHS of~\eqref{eq:remove-h}.
First, for all \ts $2\le i\le n$, we have:
\begin{equation}\label{eq:easy-FKG}
	\frac{E_{\eta_i}(g_n \ts h_{i-1})}{E_{\eta_i}(g_n)} \ \leq \ \frac{E_{\eta_i}(h_{i-1})}{E_{\eta_i}(\textbf{1})} \ = \   \frac{E_{\mu}(h_{i-1} h_i\ldots h_n)}{E_{\mu}(h_{i} \ldots h_n)}\,,
\end{equation}
where the inequality is due to the FKG inequality~(Theorem~\ref{thm:FKG}) applied to the  $\prec^\di$-increasing function~$\ts g_n$,
to the $\prec^\di$-decreasing function~$\ts h_{i-1}\ts$, and to the log-supermodular
function~$\ts\eta_i$.
Although \eqref{eq:easy-FKG} holds for \ts $i=n+1$ \ts by the same argument,
we will use the following stronger inequality instead.

\smallskip

\begin{lemma}\label{lem:strict-FKG}
	\begin{equation}\label{eq:strict-FKG}
	 \frac{E_{\eta_{n+1}}(g_n \ts h_{n})}{E_{\eta_{n+1}}(g_n)}  \ \leq \ \bigg(1-\frac{1}{t^{n+1}} \bigg) \frac{E_\mu(h_n)}{E_{\mu}(\textbf{\em 1})}.
	\end{equation}
\end{lemma}

\begin{proof}
	By a direct calculation, the claim  is equivalent to showing that
	\[     \frac{E_{\mu}(g_n) \. E_{\mu}(h_n) \, - \, E_\mu(g_n h_n) \. E_{\mu}(\textbf{1})}{E_{\mu}(g_n) \. E_{\mu}(h_n)} \ \geq \ \frac{1}{t^{n+1}}\.. \]
	Let \. $g_n',h_n':L \to \{0,1\}$ \. be given by \. $g_n'(\bv) :=  1- g_n(\bv)$ \. and \. $h_n'(\bv) := 1- h_n(\bv)$.
Then we have:
	\begin{align*}
			 g_n'(\bv) \, = \,
		\begin{cases}
			1 & \text{ if } \ v_n \. = \. 1\\
			0 & \text{ otherwise}
		\end{cases} \qquad \text{and}	\qquad		
h_n'(\bv)  \, = \,
	\begin{cases}
	1 & \text{ if } \ v_n \. = \. t\\
	0 & \text{ otherwise.}
\end{cases} 	\end{align*}
By the linearity of expectations, the claim is then equivalent to showing that
\begin{equation}\label{eq:strict-correlation}
   \frac{E_{\mu}(g_n') \. E_{\mu}(h_n') \ - \ E_\mu(g_n' h_n') \. E_{\mu}(\textbf{1})}{E_{\mu}(g_n) \. E_{\mu}(h_n)} \ \geq \ \frac{1}{t^{n+1}}.
\end{equation}
Now note that, since $n$ is a maximal element of $P$, we have:
\begin{align*}
	E_\mu(g_n') \ &= \  \big|\big\{ \bv \in L \, : \,   v_n \. = \. 1\big\}  \big| \ \geq \, 1,\\
	E_\mu(h_n') \ &= \  \big|\big\{ \bv \in L \, : \,   v_n \. = \. t\big\}  \big| \ = \,  \Omega(P\sm\{n\},t),\\
	E_{\mu}(g_n'h_n') \ &= \  \big|\big\{ \bv \in L \, : \,   v_n \. = \. 1 \. = \. t\big\}  \big| \ = \, 0,\\
	E_{\mu}(g_n)   \ &= \  \big|\big\{ \bv \in L \, : \,   v_n \. \geq  \. 2 \big\}  \big|  \  \leq \, t \. \Omega(P\sm\{n\},t),\\
	E_{\mu}(h_n)   \ &= \  \big|\big\{ \bv \in L \, : \,   v_n \. \leq  \. t-1 \big\}  \big|  \  \leq \,  \Omega(P,t) \, \leq \, t^n.
\end{align*}
The inequalities above directly imply \eqref{eq:strict-correlation}.
\end{proof}

\smallskip

\begin{proof}[Proof of Theorem~\ref{thm:log-concave-strict}]
Combining \eqref{eq:remove-g}, \eqref{eq:remove-h}, \eqref{eq:easy-FKG}, and \eqref{eq:strict-FKG},
we get:
\[  \frac{E_{\mu}(gh)}{E_{\mu}(g)}  \ \leq \  \bigg(1\.- \.\frac{1}{t^{n+1}} \bigg) \. \frac{E_\mu(h)}{E_\mu(\textbf{1})}\..\]
Using the values from \eqref{eq:expectation-as-OP}, we obtain:
\[  \frac{\Omega(P,t-2)}{\Omega(P,t-1)} \ \leq \  \bigg(1\.- \. \frac{1}{t^{n+1}} \bigg) \. \frac{\Omega(P,t-1)}{\Omega(P,t)}\..\]
The theorem now follows by substituting \. $t\gets t+1$.
\end{proof}

\smallskip

\begin{rem}
The term \. $\big(1+1/t^{n+1}\big)$ \. in~\eqref{eq:strict-FKG} is far from optimal
and can be improved in many cases.  In particular, note that in the proof of
Lemma~\ref{lem:strict-FKG} we used a separate calculation for an element~$n$
in \ts $X=[n]$.  Making this calculation for a general element \ts $x \in [n]$,
gives a lower bound with the term \. $\big(1+C/t^{\ell(x)+b(x)}\big)$, for some \ts
$C>0$. Thus \ts $x=n$ \ts  is the \emph{least optimal choice} for the lower bound,
and is made for clarity. \end{rem}

\medskip

\subsection{Kahn--Saks conjecture}\label{ss:OP-KS-mono}

The following interesting conjecture can be found in the solution to
Exc.~3.163(b) in~\cite{Sta-EC}.

\smallskip

\begin{conj}[{\rm \defn{\em Kahn--Saks monotonicity conjecture}}{}]\label{conj:KS-mon}
For a poset \ts $P=(X,\prec)$ \ts with \ts $|X|=n$ \ts elements,
the \ts \defn{scaled order polynomial} \. $\Omega(P,t)/t^n$ \. is weakly decreasing
on~$\ts\nn_{\ge 1}$.
\end{conj}

\smallskip

As Stanley points out in~\cite[Exc.~3.163(a)]{Sta-EC}, the conjecture
holds for \ts $t$ \ts large enough, since the coefficient \ts $[t^{n-1}] \ts\Omega(P,t)>0$.
Curiously, the proof is based on an elegant direct injection.
Now, to fully appreciate the power of this conjecture, let us derive from it the following
unusual extension of Theorem~\ref{thm:HP order poly}.

\smallskip
\begin{thm} \label{t:OP-via-KS-mon}
	Let \. $P=(X,\prec)$, let \. $\max(P)\subseteq X$ \. be the subset of maximal elements,
	and let \. $r:=|\max(P)|$ \. be the number of maximal elements. If Conjecture~\ref{conj:KS-mon}
	holds, then we have:
	\begin{equation}\label{eq:OP-analytic}
		\Omega(P,t) \ \geq \
		t^r \prod_{x \ts\in\ts X\sm \max(P)} \. \bigg(\frac{t}{b(x)}\,+\,\frac{1}{2}\bigg).
	\end{equation}
\end{thm}

\smallskip

Compared to~\eqref{eq:OP-HP}, the inequality \eqref{eq:OP-analytic}
adds \ts $\tfrac12$ \ts to every term in the product.
It would be interesting to prove this result unconditionally.\footnote{It would be
even more interesting to \emph{disprove} it, perhaps.}

\smallskip

\begin{proof}
	Denote
	$$
	F_m(t) \ := \ \frac{1}{t^m} \, \sum_{k=1}^t \. k^m\..
	$$
	Let us prove now, that if the Kahn--Saks monotonicity conjecture holds, then we have:
	\begin{equation}\label{eq:OP-F-ineq}
		\Omega(P,t) \ \geq \ \prod_{x \in X} \. F_{b(x)-1}(t)\ts.
	\end{equation}
To see this, first suppose that \ts $r=1$, so the poset~$P$ has a unique maximal element $x$.
Thus, $b(x)=n$ \ts in this case.
The number of order preserving functions for which $x$ has value $k$ is equal to \. $\Omega(P\sm x, t-k+1)$.
We have:	
$$
\Omega(P,t) \ = \ \sum_{k=1}^t \. \Omega(P \setminus x, k) \ \geq  \
\sum_{k=1}^t \. \Omega(P \sm x, t) \, \frac{k^{n-1}}{t^{n-1}} \ = \ \Omega(P \sm x, t) \. F_{b(x)-1}(t),
$$
where the inequality follows from the conjectured monotonicity. When \ts $r\ge 2$,
the rest of the proof of~\eqref{eq:OP-F-ineq}
follows verbatim the proof of Theorem~\ref{thm:HP order poly}.
	
Now note the following bounded version of the \defng{Faulhaber's formula}:
$$
F_m(t) \, \geq \, \frac{t}{m} \. + \. \frac{1}{2} \quad \text{for all} \ \ m \geq 2.
$$
This inequality is well-known and can be easily proved by induction.
Substituting it into~\eqref{eq:OP-F-ineq}, gives the result.
\end{proof}

\smallskip

In support of this conjecture we prove the following partial result.

\smallskip

\begin{prop}
\label{prop:scaled}
Let \ts $P=(X,\prec)$ \ts be a poset with \ts $|X|=n$ \ts elements,
and let \ts $k, \ts t\in \nn_{\ge 1}$\ts. Then:
$$
\frac{1}{t^n} \, \Omega(P,t) \, \geq \, \frac{1}{(kt)^n} \, \Omega(P,kt)\ts.
$$
Moreover, there is an injection which shows that the function \. $\Omega(P,t)\ts k^n \ts - \ts \Omega(P,kt) \in \SP$.
\end{prop}

\begin{proof}
Let \. $f \in \BOm(P,kt)$. Consider an increasing function \ts $g \in \BOm(P,t)$ \ts given by
$$g(x) \, := \, \left \lfloor \frac{f(x)-1}{k}\right\rfloor \. + \. 1,
$$
and let \ts $\beta: X \to \{0,1,\ldots,k-1\}$ \ts be given as the residue of~$f(x)$ modulo~$k$.
It is clear that the pair \ts $(g,\beta)$ \ts uniquely determines~$f$.
Then \. $\Omega(P,t)\ts k^n \ts - \ts\Omega(P,kt)$ \. is the number of pairs \ts $(g,\beta)$,
such that the map \ts $h:X \to [n]$ \ts given by \ts $h(x) := k(g(x)-1) + \beta(x) +1$ \ts
is not a linear extension, i.e.\ if \ts $h(x) > h(y)$ \ts for some \ts $x\prec y$.
The last condition can be verified in polynomial time, proving that the difference is in~$\SP$.
\end{proof}

\medskip

\subsection{Reverse monotonicity}\label{ss:OP-reverse-mono}
The following result at first appears counterintuitive until one
realizes that it's trivial asymptotically, when \ts $t\to \infty$.
Just like the Kahn--Saks monotonicity conjecture, the small values
of~$t$ is where the difficulty occurs.

\smallskip

\begin{thm}\label{t:reverse-width}
Let \ts $P=(X,\prec)$ \ts be a finite poset of width~$w$. Then the function \.
$\Omega(P,t)/t^w$ \. is weakly increasing on all \. $t \in \nn_{\ge 1}$\ts.
\end{thm}

\smallskip

The proof is based on yet another application of the FKG inequality
in the following lemma of independent interest.

\smallskip

\begin{lemma}\label{lem:KS-mon}
Let \ts $P=(X,\prec)$ \ts be a finite poset and let \ts $t \geq k \ge 1$ \ts be positive integers.
Then, for every minimal element $x$ of~$P$, we have:
	\begin{equation}\label{eq:KS-mon}
	\frac{\Omega(P,k)}{\Omega(P,t)} \ \leq \  \.  \frac{\Omega(P\setminus x,k)}{\Omega(P\setminus x,t)}\..
\end{equation}
\end{lemma}

\smallskip

\begin{proof}
Let \ts $Y=\{x\}$ \ts and \ts $Z=X \setminus \{x\}$.
When $x$ is incomparable to every element of $Z$, we have
\begin{equation}\label{eq:OP-k-t}
	  \frac{\Omega(P,k)}{\Omega(P,t)} \ = \ \frac{k\. \Omega(P\setminus x,k)}{t\. \Omega(P\setminus x,t)}\.,
\end{equation}
and the result follows.

Thus, without loss of generality we can assume that \ts $x \prec z$ \ts for some \ts $z \in Z$.
	Let \. $g,h: L \to \{0,1\}$ \. be given by
$$
	g(\bv)  \, := \, \left\{
\aligned\. 1 & \quad \text{if} \ \ v_x  \. \leq \. v_z \ \ \. \text{for all} \ \ z\in B(x), \, z\ne x \\
	\. 0 & \quad\text{otherwise}, \endaligned \right.
$$
$$
 h(\bv) \, := \, \left\{ \aligned \. 1 & \quad \text{if} \ \  v_z \leq k \ \ \. \text{for all} \ \  z \in Z\\
	\. 0 & \quad \text{otherwise}. \endaligned\right.
$$
To show that \ts $g$ \ts is \ts $\prec^\di$-decreasing,
it suffices to check that,
\begin{equation*}
g(\bw)\ts =\ts 1 \quad \text{and} \quad \bv \prec^\di \bw \quad \Longrightarrow \quad  g(\bv)\ts = \ts 1.
\end{equation*}
Note that for every \ts $z \in Z$ \ts such that \ts $x \prec z$, we have:
\[  v_x \ \leq \  w_x \ \leq \ w_z \ \leq \ v_z, \]
where the first and the third inequality is because \ts $\bv \prec^\di \bw$\ts,
and the second inequality is because \ts $g(\bw)=1$.
This implies that \ts $g(\bv)=1$, and thus \ts $g$ \ts is
a $\prec^\di$-decreasing function.

Similarly, to show that \ts $h$ \t is $\prec^\di$-increasing,
it suffices to check that
\begin{equation*}
h(\bv)\ts =\ts 1 \quad \text{and} \quad \bv \prec^\di \bw \quad \Longrightarrow \quad  h(\bw)\ts = \ts 1.
\end{equation*}
Note that for every \ts $z \in Z$, we have:
\[  w_z \ \leq \  v_z \ \leq \ k, \]
where the first  inequality is because \ts $\bv \prec^\di \bw$\ts, and the second inequality is because \ts $h(\bv)=1$.
This implies that \ts $h(\bw)=1$, and thus \ts $h$ \ts is a $\prec^\di$-increasing function.

		Now observe that the um \ts
		$E(gh)$ \ts counts \ts $\bv \in L$ \ts for which \ts
		$v_x \leq k$. Indeed, by the assumption there exist \ts $z \in Z$,
such that \ts $x \prec z$.  This implies
		\[ v_x \ \leq \ v_z \ \leq \ k, \]
		where the first inequality is because \. $g(\bv)=1$\., and the second inequality is because \. $h(\bv)=1$.
		Thus, \ts $E(gh)$ counts the number of order preserving maps \. $f: X \to [k]$, i.e.  $E(gh)=\Omega(P,k)$.
		It is also straightforward to verify that
		\begin{align*}
			E(g)  \, = \, \Omega(P,t), \quad E(h) \, = \,  t \. \Omega(P \setminus x,k) \quad \text{and}
\quad E(\textbf{1}) \, = \, t \. \Omega(P\setminus x,t).
		\end{align*}
The lemma now  follows from the FKG inequality~(Theorem~\ref{thm:FKG}).
\end{proof}

\smallskip

\begin{proof}[Proof of Theorem~\ref{t:reverse-width}]
Let \ts $H$ \ts be a maximal antichain in \ts $P$ \ts of width~$w$.
Note that Lemma~\ref{lem:KS-mon} can also be applied to maximal
elements~$x$ of~$P$, by considering the dual poset \ts $P^\ast$.
Now, by consecutively removing elements in \ts $X \setminus H$,
by Lemma~\ref{lem:KS-mon}, we get
	\begin{equation*}
	\frac{\Omega(P,k)}{\Omega(P,t)} \ \leq \  \.  \frac{\Omega(P',k)}{\Omega(P',t)} \ = \  \frac{k^{w}}{t^{w}}\,,
\end{equation*}
where \ts $P'=P|_H$ \ts is the subposet of $P$ restricted to~$H$.
This implies the result.
\end{proof}

\smallskip

We conclude with another conjecture motivated by~\eqref{eq:OP-k-t} in the proof of Lemma~\ref{lem:KS-mon}.

\smallskip

\begin{conj}\label{conj:KS-FKG}
	Let \ts $P=(X,\prec)$ \ts be a finite poset, and let
\. $t \geq k\ge 1$ \. be positive integers.
	Then there exists \ts $x \in X$, such that
	\begin{equation}\label{eq:KS-FKG}
	  \frac{\Omega(P,k)}{\Omega(P,t)} \ \geq \  \frac{k\. \Omega(P\setminus x,k)}{t\. \Omega(P\setminus x,t)}.
	\end{equation}
\end{conj}

\smallskip

\begin{prop}
Conjecture~\ref{conj:KS-FKG} \ts implies \ts
Conjecture~\ref{conj:KS-mon}.
\end{prop}

\smallskip

The proof of this proposition follows the proof of the theorem above.

\medskip

\subsection{Daykin--Daykin--Paterson inequality}\label{ss:OP-Graham}
Let \ts $P=(X,\prec)$ \ts on \ts $|X|=n$ elements.
Fix an element \ts $x\in X$ \ts and an integer \ts $t\ge 1$.
Denote by \ts $\Om(P,t; \ts x,a)$ \ts the number of order preserving
maps \ts $g: X\to [t]$, such that \ts $g(x)=a$.  The following result 
resolves a conjecture by Graham in \cite[p.~129]{Gra}, made 
by analogy with Stanley's inequality~\eqref{eq:Sta}.  

\smallskip

\begin{thm}[{\em {\rm Daykin, Daykin and Paterson} \cite{DDP}, 
{\rm formerly} \defn{\rm Graham's conjecture}}{}]\label{conj:Graham}
Let \ts $P=(X,\prec)$ \ts be a finite poset, and
let \ts $x\in X$, let \ts $a, \ts t\in \nn_{\ge 1}$, and suppose \ts $1< a < t$.
Then:
\begin{equation}\label{eq:Graham}
	\Om(P,t; \ts x,a)^2 \, \ge \, \Om(P,t; \ts x,a+1) \. \cdot \. \Om(P,t; \ts x,a-1).
\end{equation}
\end{thm}

\smallskip

The proof in \cite{DDP} is based on a direct injection (see~$\S$\ref{ss:finrem-Graham}).  
Curiously, we can use Theorem~\ref{thm:log-concave-strict} to show that \eqref{eq:Graham}  
holds asymptotically.

\smallskip

\begin{cor}\label{t:Graham-asy}
Let \ts $P=(X,\prec)$ \ts be a finite poset, and let \ts $x\in X$.
Then, for every integer \ts $a\ge 1$,
	there exists  \. $T(P,x,a)>0$, such that  for all \. $t > T(P,x,a)$, we have:
	 $$
	 \Om(P,t; \ts x,a)^2 \, \ge \, \Om(P,t; \ts x,a+1) \. \cdot \. \Om(P,t; \ts x,a-1).
	 $$
	 Furthermore, if $x$ is  incomparable to any other element of $P$, then the inequality above is strict for sufficiently large  $t$.
\end{cor}

\begin{proof}
Fix an integer \ts $a\ge 1$.
First note that, if  $x$ is incomparable to every other element  in $P$, then equality in fact occurs in the inequality \eqref{eq:Graham}.
So we can assume that $x$ is comparable to some other element $y$ of $P$, and
we will  further assume that $y \prec x$, as the proof for the other case is analogous.
Denote by \.
$D:=\{y\in X \.:\.y \preccurlyeq x\}$ \. the lower order ideal of~$x$, and
let \ts $d:=|D|$.
Note that \. $D-x$ \. is a non-empty set by assumption.
Now observe that \. $\Om(P,t; \ts x,a)$ \. is a polynomial in~$\ts t$ \ts
	with the leading term
$$
\Omega(D-x,a)  \, \frac{e(P\sm D)}{(n-d)!} \ t^{n-d}\..
$$
Indeed, for every $\prec$-increasing function \ts $g: X\to [t]$,
we have \ts $g(y)\le a$ \ts for all \ts $y\in D$, which explains
the term \ts $\Omega(D-x,a)$.  For the remaining elements
\ts $z\in X\sm D$, we have no such restrictions as \ts $t\to \infty$,
and the number of such functions is asymptotically \ts $\sim e(P\sm D) \. t^{n-d}/(n-d)!$

Therefore, as \ts $t\to \infty$, the leading coefficient  of the polynomial
$$\Om(P,t; \ts x,a)^2 \, - \, \Om(P,t; \ts x,a+1) \. \cdot \. \Om(P,t; \ts x,a-1)
$$
is equal to
\[
\Big[\Omega(D-x,a)^2 \. - \.  \Omega(D-x,a+1) \.
\Omega(D-x,a-1) \Big] \, \frac{e(P-D)}{(n-d)!}\..  \]
This is strictly positive by Theorem~\ref{thm:log-concave-strict} (note that $D-x$ is a non-empty poset by assumption),
which implies the result.
\end{proof}

\medskip

\section{Bounds on the $q$-analogue} \label{s:q-analogue}

In this section we study the $q$-order polynomial generalization.
First, we present Reiner's short proof of the Bj\"orner--Wachs inequality.
Then, we give $q$-analogue of Shepp's inequality and study its consequences.

\subsection{Reiner's inequality}\label{ss:q-analogue-Reiner}
Most recently, Vic Reiner shared with us the following elegant approach
to the Bj\"orner--Wachs inequality which we reproduce with his
permission.\footnote{Vic Reiner, personal communication, March~17, 2022. }

\smallskip

\begin{thm}[{\rm Reiner, 2022}] \label{t:HP-q}
Let \ts $P=(X,\prec)$ \ts be a poset with \ts $|X|=n$ \ts elements.
Denote by \. $\cR(P)$ \. the set of all weakly order-preserving maps \.
$g: X\to \nn$, i.e.\ \. $g(x)\le g(y)$ \. for all \. $x\prec y$.
Let \. $|g| \ts := \ts \sum_{x\in X} \ts g(x)$.
Then:
\begin{equation}\label{eq:HP-q}
\sum_{g\in \cR(P)} \. q^{|g|} \ \geqslant_q  \  \. \prod_{x \ts \in X} \, \frac{1}{1-q^{b(x)}}\,,
\end{equation}
where the inequality between two power series is coefficient-wise.
\end{thm}

\smallskip

Theorem~\ref{t:HP} follows from \defng{$P$-partition theory} of Stanley,
see~\cite[$\S$3.15]{Sta-EC}.  Indeed, recall that
$$
\sum_{g\in \cR(P)} \. q^{|g|} \, = \, \frac{F_P(q)}{(1-q)(1-q^2)\cdots (1-q^n)}\,,
$$
such that \ts $F_P(1)=e(P)$.  Here \ts $F_P(q)$ \ts denotes the sum of \ts
$q^{\maj(f)}$ \ts over all \ts $f\in \Ec(P)$, see~\cite{Sta-EC} for the
details.\footnote{The notation in~\cite{Sta-EC} is different but equivalent;
we change it for simplicity since the \defng{major index} plays only tangential
role in this paper.  There is also a minor subtlety here, that
\defng{Stanley's P-partition theory} needs to be applied to a \emph{natural
labeling} of~$X$, cf.~$\S$\ref{s:OP-inj}. }
Taking \ts $0< q <1$, multiplying both sides of~\eqref{eq:HP-q} by
\. $(1-q)(1-q^2)\cdots (1-q^n)$, and taking the limit \. $q\to 1-$, gives the
Bj\"orner--Wachs inequality~\eqref{eq:HP}.

\smallskip

\begin{proof}[Proof of Theorem~\ref{t:HP-q}]
Interpret the RHS of~\eqref{eq:HP-q} as the GF for maps \. $g\in \cR(P)$ \. which
are obtained as a nonnegative integer linear combination of the characteristic
functions of upper order ideals:
$$g \.= \, \sum_{x\in X} \, m(x) \.\chi_{B(x)}\., \ \ \. \text{where} \ \ m(x) \in \nn \ \ \text{for all} \ \ x\in X.
$$
Note that characteristic functions \ts $\chi_{B(x)}$ \ts are linearly independent
because in the standard basis \ts $\chi_{y}$, \ts $y\in X$, the transition matrix is unitriangular.
Since now \. $|g| = \sum_{x\in X} \. m(x) \. b(x)$, the result follows immediately.
\end{proof}

\smallskip

\begin{ex}
Consider a poset \ts $P=(X,\prec)$ \ts with \ts $X=\{a,b,c,d\}$ \ts
and \ts  $a\prec \{b,c\} \prec d$, so \ts $P\simeq C_2\times C_2$.
The RHS of~\eqref{eq:HP-q} as in the proof
above is the GF for \ts $g\in \cR(P)$, such that \ts $g(a)+g(d) \ge g(b) + g(c)$.
Not all \ts $g\in \cR(P)$ \ts satisfy this property, e.g.\ $g(a)=0$, \ts
$g(b)=g(c)=g(d)=1$ \ts does not.
\end{ex}
\smallskip

\begin{rem}\label{r:G-M}
In principle, there is a way to convert the natural injection as in
the proof above into an injection as in Proposition~\ref{p:HP-injection}.
The idea is to make the multiplication by \. $(1-q)\cdots (1-q^n)$ \.
to be effective by using the \defng{involution principle} of Garsia
and Milne~\cite{GM}.  See also~\cite{Gre} which comes closest in this
special case.  Note that the resulting maps tend to be hard
to compute, sometimes provably so, see e.g.~\cite{KP}.
\end{rem}

\smallskip

\subsection{$q$-order polynomial} \label{ss:q-analogue-our}
For an integer \ts $t\ge 1$, define
$$\Omega_q(P,t) \, : = \, \sum_{g} \ q^{|g|-n}
$$
where the summation is over all order preserving maps \. $g: X\to [t]=\{1,\ldots,t\}$,
i.e.\ maps which satisfy \. $g(x)\le g(y)$ \. for all \. $x \prec y$.
This is the \defn{$q$-order polynomial}
corresponding to poset~$P$, see e.g.~\cite{Cha}.  Let us emphasize that here~$q$
is a formal variable, while \ts $t\ge 1$ \ts is an integer.

\smallskip

\begin{thm}[{\rm {\em{\defn{$q$-analogue of Shepp's inequality}}}}]\label{t:FKG-OP-q}
Let \ts $A$ \ts be the collection defined in \S\ref{ss:OP-induction-prelim}, and let \ts $C,C'$ \ts
be $Y$-minimizing collections w.r.t.\ partition \ $X=Y\sqcup Z$. Then,
\[
\Omega_q \big(C  \cap   C' \cap A,\ts t\big)  \. \cdot \.  \Omega_q\big(A,\ts t\big)
\quad \geqslant_q \quad \Omega_q \big(C \cap A,\ts t\big)  \. \cdot \.  \Omega_q\big( C' \cap A,\ts t\big)\.,	
\]
where the inequality holds coefficient-wise as a polynomial in~$q$,
for all integer \ts $t\ge 1$.
\end{thm}

\smallskip

The proof follows the original proof in~\cite{She}, with  the following \defng{$q$-FKG inequality} by Bj\"orner~\cite{Bjo}.
Let \. $L:=(L,\prec^\di)$ \. be a distributive lattice.
A function \. $r:L \to \Rb_{\geq 0}$ \. is called \ts \defn{modular} \ts if
\[ r(a) \. + \.   r(b) \ = \  r(a \wedge b) \. + \.  r(a \vee b) \quad \text{ for every } \. a,b \in L. \]

\smallskip

\begin{thm}[{\rm {\em{\defn{$q$-FKG inequality}}}, \cite[Thm~2.1]{Bjo}}]\label{thm:q-FKG}
Let \ts $\ll =(L,\prec)$ \ts be a finite distributive lattice,
let \. $\mu: L \to \Rb_{\geq 0}$ \. be a log-supermodular function, and let
\. $r:L \to \Rb_{\geq 0}$ \. be a modular function.
	Then, for every pair of $\prec^\di$-decreasing functions \. $g,h: L \to \Rb_{\geq 0}$, we have:
	\[ E_q(\textbf{\em 1}) \. E_q(gh) \ \geqslant_{q} \  E_q(g) \. E_q(h), \]
	where the inequality holds coefficient-wise as a polynomial in $q$, \.
	where
	\[ E_{q}(g) \, = \, E_{q}(g;\mu,r) \ := \  \sum_{x \in L} g(x) \. \mu(x) \. q^{r(x)},
\]
and \. $\textbf{\em 1}:L \to \Rb$ \. is given by \. $\textbf{\em 1}(x)=1$ \. for all \ts $x \in L$.
\end{thm}

\smallskip

\begin{rem}\label{r:Bjorner}
Note that Theorem 2.1 in \cite{Bjo} assumes that \. $r:L \to \Rb_{\geq 0}$  \.
is the \emph{rank function} of the lattice~$L$. It is however straightforward
to show that the same proof still works when applied to any modular function~$\ts r$.
\end{rem}

\smallskip

\begin{proof}[Proof of Theorem~\ref{t:FKG-OP-q}]
	The proof follows the same argument as in  the proof of Lemma~\ref{t:FKG order polynomial},
	with the FKG inequality being replaced with Theorem~\ref{thm:q-FKG} applied to the modular function
	 \. $r:L \to \Rb_{\geq 0}$ \.  given by \. $r(\bv) \. := \. \sum_{x \in X} v_x$\..
\end{proof}

\smallskip

\begin{cor}\label{cor:main OP-q}
	Let \. $P=(X,\prec)$ \. be a poset, and let \. $x,y\in X$ \. be minimal elements.
	Then, for all \ts $t\in \nn_{\ge 1}$ \ts and \ts $q\in \rr_+$, we have:
	\[ {\Omega_q\big(P,\ts t\big)} \. \cdot \. {\Omega_q\big(P \sm \{x,y\},\ts t\big)}
\ \geq \  {\Omega_q\big(P \sm x, \ts t\big)}  \. \cdot \.  {\Omega_q\big(P\sm y,\ts t\big)}.
	\]
\end{cor}

\begin{proof}
Denote \. $(n)_q \ts :=  \ts 1+q+\ldots + q^{n-1}$.
	Let $C$ and $C'$ be as in \eqref{eq:colC}, and $A$ be as in \eqref{eq:colA}.
	Observe that
	\begin{alignat*}{2}
		& \Omega_q(C \cap C' \cap A,t) \ = \ \Omega_q(P,t), \qquad  &&  \Omega_q(C' \cap A,t) \ = \ q \ts (t)_q \. \Omega_q(P\sm x,t),\\
		&  \Omega_q(C \cap A,t) \ = \ q \ts (t)_q \. \Omega_q(P \sm y,t), \qquad  &&  \Omega_q(A,t) \ = \ q^2 \ts (t)_q^2 \.\ts \Omega_q(P \sm \{x,y\},t).
	\end{alignat*}
	The conclusion of the lemma now follows from  Theorem~\ref{t:FKG-OP-q}
	and the equation above.
\end{proof}

\smallskip

\begin{rem}\label{r:q-analogue}
Note that our proof does not show that the inequality in Corollary~\ref{cor:main OP-q} holds
 coefficient-wise as a polynomial in~$q$,
 since the derivation involves canceling the term \ts $q\ts (t)_q$.
It remains to be seen if a $q$-analogue of Theorem~\ref{thm:HP order poly} exists,
which hinges on finding an appropriate $q$-analogue for Lemma~\ref{l:main order polynomial}.
\end{rem}

\smallskip

We also have the following \defng{$q$-log-concavity} for order polynomials.

\smallskip

\begin{cor}\label{c:q-log-concavity}
	Let \ts $P=(X,\prec)$ \ts be a finite poset.
Then, for every integer \ts $t \geq 2$, we have:
\[ \Omega_q(P,t)^2  \  \geqslant_q \ \Omega_q(P,t+1) \. \cdot \. \Omega_q(P,t-1),\]
where the inequality holds coefficient-wise as a polynomial in~$q$.
\end{cor}

\begin{proof}
		The proof follows the same argument as in  the proof of Theorem~\ref{thm:log-concave},
	with the FKG inequality being replaced with Theorem~\ref{thm:q-FKG} applied to the modular function
	\. $r:L \to \Rb_{\geq 0}$ \.  given by \. $r(\bv) \. := \. \sum_{x \in X} v_x$\..
\end{proof}

\bigskip

\section{Bounding the order polynomial by injection} \label{s:OP-inj}

\smallskip

Let \ts $P=(X,\prec)$ \ts be a poset with \ts $|X|=n$ \ts elements.
Denote by \ts $\BOm(P,t)$ \ts the set of order preserving maps \ts $P \to [t]$, so that \. $\Om(P,t)=|\BOm(P,t)|$.
Fix a \defn{natural labeling} of $X$, i.e.\ write \. $X=\{x_1,\ldots,x_n\}$, where \. $i<j$ \. for all \. $x_i \prec x_j$.

For a sequence \. $(a_1,\ldots,a_k)$ \. of distinct integers, a \defn{standardization}
is a permutation \. $\si=(\si_1,\ldots,\si_k)\in S_k$ \. with integers in the same
relative order:
$$
a_i < a_j \ \Leftrightarrow \ \si_i < \si_j \quad \text{for all} \quad 1\le i < j \le k\ts.
$$
For example, the standardization of \. $(4,7,6,3)$ \. is \. $(2,4,3,1)\in S_4$.

\smallskip

\begin{proof}[Proof of Theorem~\ref{thm:order_le}]
We construct an injection
$$\Psi: \, \Ec(P) \ts \times \ts [t]^n \. \to \ \BOm(P,t) \ts \times \ts S_n\ts.
$$
One can think of \. $[t]^n$ \. as an \emph{ordered set partition}
$$
[n] \ = \ B_1 \. \sqcup \. \ldots \. \sqcup B_t\.,
$$
where \. $B_i\subseteq [n]$ \. can be empty.  We use \. $\be=(B_1,\ldots,B_t)$ \. to denote this ordered
set partition.

Let \. $f \in \Ec(P)$ \. be a linear extension, and \. $\be=(B_1,\ldots,B_t)$ \.
be an ordered set partition as above.  Denote \. $b_i := |B_i|$,
where \. $1 \le i \le t$.  Let \. $\al = (a_1,\ldots,a_n)\in [t]^n$ \.
be a weakly increasing sequence
$$
\big(1,\ldots, 1, \ 2,\ldots, 2, \ \ \ldots \ \  , \ t, \ldots, t\big) \ \ \,
\text{with \. $b_i$ \. copies of \. $i$, \ for all \. $1\le i \le t$}\ts.
$$
By abuse of notation, we also use \ts $\al$ \ts to denote a function \. $\al: [n]\to [t]$ \. given by \.
$\alpha(i):=a_i$.

Define a function \. $g: \ts X \to [t]$ \. as \. $g(x_i) := \alpha\bigl(f(x_i)\bigr)$,
so that elements \. $f^{-1}(1)$, \ldots, $f^{-1}(b_1)$ are assigned value~$1$, elements \.
$f^{-1}(b_1+1)$, \ldots, $f^{-1}(b_1+b_2)$ are assigned value~$2$, etc. Observe that \.
$g \in \BOm(P,t)$ \. since~$\ts f$ \ts is increasing with respect to the poset order, and~$\ts \al$
is a weakly increasing function.

Next, define a permutation \. $\sigma\in S_n$ \. as follows. For each~$i$,
let \. $g^{-1}(i) = \big\{x_{i_1},\ldots,x_{i_k}\big\}$, where \. $i_1<\ldots<i_k$ \. and \ts $k=b_i$ \ts by construction.
Let \ts $s^{(i)} \in S_k$ \ts be the standardization of the sequence \. $\bigl(f(x_{i_1}),\ldots, f(x_{i_k})\bigr)$.
Now, rearrange the elements in \ts $B_i$ \ts according to~$s^{(i)}$, obtaining a sequence \.
$\gamma_{i}$ \.
whose standardization is~$s^{(i)}$.
The permutation $\sigma$ is then obtained by concatenating the resulting sequences, i.e.
$\sigma:=\gamma_{1} \gamma_{2}\ldots \gamma_{t} \in S_n$.
Finally, define \. $\Psi(f,\beta) := (g, \sigma)$.

\smallskip

To prove that \ts $\Psi$ \ts is an injection, we construct an inverse map \ts $\Psi^{-1}$.
Let \. $g \in \Omega(P,t)$ \. and \. $\sigma \in S_n$. Denote  \. $c_i :=|g^{-1}(i)|$,
for all \. $1\le i \le t$.  Let \. $\tau\in [t]^n$ \. be the sorted
sequence of values that the function~$g$ takes, i.e.\
$$\tau \ := \ \big(1,\ldots, 1, \ 2,\ldots, 2, \ \ \ldots \ \  , \ t, \ldots, t\big) \ \ \, \text{with \.
$c_i$ \. copies of \. $i$, \ for all \. $1\le i \le t$}\ts.
$$
Note that \ts $\tau$ \ts is the weakly increasing.
%
For each $i$, let
$$
C_i \ := \ \big\{\si_{c_1\ts+\ts\ldots \ts+\ts c_{i-1}\ts+\ts1} \., \. \ldots \., \. \si_{c_1\ts+\ts\ldots \ts+\ts c_i}\big\}
$$
consisting of a block of size $c_i$ of entries from~$\si$.
Denote by \. $\pi=(C_1,\ldots,C_t)$ \. the resulting ordered partition.

Finally, define a function \.
$h: X\to [n]$ \. obtained by rearranging the values on the $c_i$ elements in~$g^{-1}(i)$
according to the ordering in \.
$$
\big(\sigma_{c_1\ts+\ts\ldots \ts+\ts c_{i-1}\ts+\ts1} \., \. \ldots \., \.  \sigma_{c_1\ts+\ts\ldots \ts+\ts c_i}\big),
$$
i.e., so that their standardizations are the same permutations.
Let us emphasize that \ts $h$ \ts
is not necessarily a linear extension for general \ts $(g,\si)$ \ts as above.

Now take \. $\Psi^{-1} := (h,\pi)$, and observe that
$$
\Psi^{-1}\bigl(\Psi(f,\be)\bigr) \, = \, (f,\be)
$$
by construction.  This completes the proof.
\end{proof}

\smallskip

\begin{ex}
Let us illustrate the construction of \. $\Psi(f,\be)=(g,\si)$ \. in the proof above.
Let \. $P=(X,\prec)$ \. be a poset on \ts $n=7$ \. elements as in
Figure~\ref{f:injection}, where
\. $X=\{x_1,\ldots,x_7\}$ \. with the partial order~$\prec$ increasing downwards.
Note that we chose a natural labeling, see above.

Suppose \ts $t=3$.  Let \ts $f\in \Ec(P)$ \ts be a linear extension as in the figure, and let
\. $\beta=(B_1,B_2,B_3)$, where \. $B_1=\{2,3,7\}$,
$B_2=\{4,6\}$ \. and \. $B_3 = \{1,5\}$. Then we have \. $\al = (1,1,1,2,2,3,3)$ \. and
the order preserving function \ts $g$ \ts is given as in the figure.  Then,
standardize the values \.
$$\aligned
& \hskip2.cm\bigl(f(x_1),f(x_3),f(x_5)\bigr) = (2,1,3) \  \longrightarrow \ (2,1,3)\.,\\
& \bigl(f(x_2),f(x_7)\bigr) = (4,5) \   \longrightarrow \ (1,2)\., \qquad
\bigl(f(x_4),f(x_6)\bigr) = (6,7) \  \longrightarrow \ (1,2).
\endaligned
$$
Permute the elements within \. $B_1,B_2,B_3$ \. accordingly to get \. $\gamma_1 = (3,2,7)$, \. $\gamma_2=(4,6)$ \.
 and \. $\gamma_3 = (1,5)$.  Concatenating these, we obtain \. $\sigma = (3,2,7,4,6,1,5)\in S_7$.

\smallskip

In the opposite direction, let \. $g'\in \BOm(P,3)$ \. be as in Figure~\ref{f:injection}, and let \.
$\si=(7,1,3,2,5,4,6)$.
Then \. $\tau=(1,1,2,2,3,3,3)$, so \. $c_1=c_2=2$ \. and \. $c_3=3$.
This gives \. $C_1=\{1,7\}$, \. $C_2=\{2,3\}$ \. and
\. $C_3=\{4,5,6\}$.  The corresponding reduced permutations are then \. $(2,1)$, \. $(2,1)$ \.
and \. $(2,1,3)$, respectively, giving a map \. $h: X\to [t]$.  Finally, note that \. $h\notin \Ec(P)$ \.
in this case.
\end{ex}

\begin{figure}[hbt]
\begin{center}
\includegraphics[width=16.5cm]{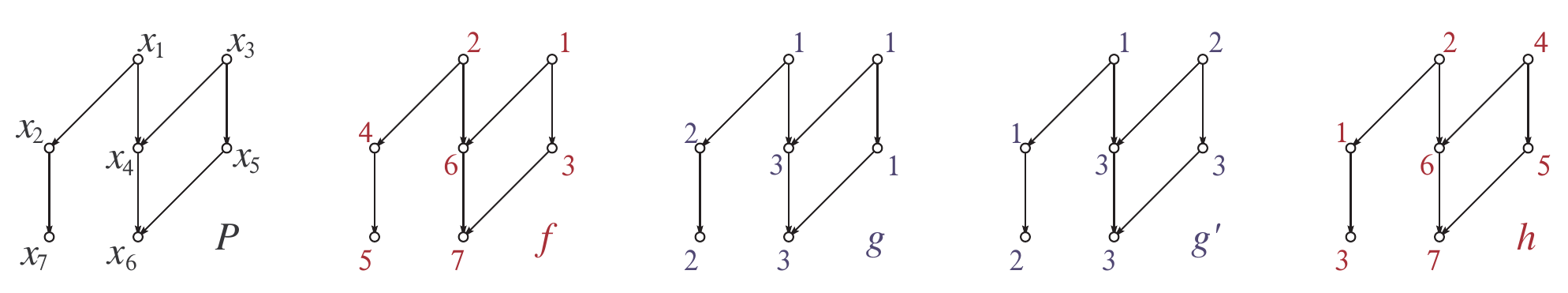}
\end{center}
\vskip-.3cm
\caption{An example of the injection $\Psi$ and the inverse map $\Psi^{-1}$. }\label{f:injection}
\vskip.3cm
\end{figure}

\smallskip

\begin{proof}[{Proof of Theorem~\ref{t:OP-SP}}]
In the notation of the proof above, let \. $(g,\si) \in \BOm(P,t) \times S_n$ \.
and let \. $(h,\pi) =  \Psi^{-1}(g,\si)$.   By construction, we have \.
$(h,\pi) \in \Ec(P) \times [t]^n$ \.  if and only if \. $h \in \Ec(P)$.
Thus, the function  \. $\ze(P,t)$ \. is equal to the number of  \.
$(g,\si) \in \BOm(P,t) \times S_n$ \. such that \ts $h \notin \Ec(P)$.
Since \ts $\Psi^{-1}$ \ts can be computed in polynomial time, this
implies the result.
\end{proof}

\smallskip

\begin{ex}\label{e:HP-power}
Let \. $P= A_{n}$ \. be an antichain on \ts $n$ \ts elements.
Then we have \. $e(A_n)=n!$ \. and \. $\Omega(A_n,t)=t^n$.
In this case both~\eqref{eq:OP-HP} and~\eqref{eq:OP-gen} are equalities.
Similarly, let \ts $P=C_n$ \ts be a chain of $n$ elements. Then we have \.
$e(C_n)=1$ \. and \. $\Omega(P,t) = \binom{t+n-1}{n}$.  In this case,
the lower bound~\eqref{eq:OP-HP} is slightly better than~\eqref{eq:OP-gen}.
\end{ex}

\smallskip

In a different direction, here are equality conditions for~\eqref{eq:OP-gen} in
Theorem~\ref{thm:order_le}.

\smallskip

\begin{cor}\label{c:OP2-equality}
Let \ts $P=(X,\prec)$ \ts be a poset on \ts $|X|=n$ \ts elements.  Then
$$
\Omega(P,t) \, = \, e(P) \. \frac{t^n}{n!}
$$
for some \. $t \in\nn_{\ge 1}$ \. \underline{if and only if} \. $P=A_n$ \ts
is an \ts $n$-antichain.
\end{cor}

\begin{proof}
The ``if'' part is clear.
\ts
For the ``only if'' part, suppose that \ts $P \neq A_n$.
Then there are \ts $x_i, x_{i+1}\in X$ \ts such that \ts $x_i\prec x_{i+1}$,
and \ts $x_{i+1}$ covers~$x_i$.  Without loss of generality,
we can assume that \ts $x_i$ \ts is a minimal element.

It follows from the proof of Theorem~\ref{thm:order_le},
that equality in~\eqref{eq:OP-gen} holds \ts if and only if \ts
$\Psi$ is a bijection. In particular,  for all  \.$(g,\si) \in \BOm(P,t) \times S_n$ \.
the map \ts $h$ \ts in \. $\Psi^{-1}(g, \si) = (h,\pi)$ \. must be a linear extension.
Now take an order preserving map~$g$, such that \. $g(x_1)=\ldots=g(x_{i+1})=1$,
and \ts $\sigma = (i,i+1)$. Then we have \ts $h(x_i)=i+1$ \ts and \ts $h(x_{i+1})=i$,
so that \ts $h\notin \Ec(P)$ \ts is not a linear extension. Thus, map~$\Psi$ is not a
bijection in this case.
This completes the proof.
\end{proof}

\smallskip

\begin{ex}\label{e:OP-Stanley}
Let \. $P_n = C_1 \oplus A_{n-1}$ \. be an ordered tree poset consisting of one minimal element
and $(n-1)$ maximal elements.  Then \ts $e(P)=(n-1)!$ \ts and the bound~\eqref{eq:HP} is an equality.
Observe that
$$
\Omega(P_n,t) \ = \ 1^{n-1}\. + \. 2^{n-1}\. + \. \ldots \. + \.  t^{n-1} \ = \ \frac{t^n}{n} +  \frac{t^{n-1}}{2} +   O(t^{n-1}).
$$
The general inequality~\eqref{eq:OP-gen} gives \. $\Omega(P_n,t) \ts \ge \ \frac{t^n}{n}$\ts,
while~\eqref{eq:OP-HP} gives a stronger bound:
$$
\Omega(P_n,t) \ \ge \ \frac{1}{n} \. \bigl(t^n + \.t^{n-1}\bigr).
$$
Asymptotically, both lower bounds are not tight in the second order term.
Compare this to~\eqref{eq:OP-F-ineq} which conjecturally gives a sharp bound.
\end{ex}

\smallskip

\begin{rem}\label{r:OP-injection}
Neither of the Theorems~\ref{thm:HP order poly} and~\ref{thm:order_le} imply each other.
Note that the leading coefficient of \ts $\Omega(P,t)$ \ts is \ts $e(P)/n!$,
and the Bj\"orner--Wachs inequality~\eqref{eq:HP} is an equality only for
ordered forests (Proposition~\ref{p:HP-forest}). Thus, for large values of~$t$,
the lower bound in Theorem~\ref{thm:order_le} asymptotically better.

On the other hand, the lower bound in Theorem~\ref{thm:order_le} cannot
be improved to \. $t^r\ts (t+1)^{n-r}\ts e(P)/n!$ \. as Theorem~\ref{thm:HP order poly}
might suggest.  Indeed, for the poset $P_n$ as in the example above, we have:
$$\Omega(P_n,2) \ = \ 1^{n-1}+2^{n-1} \ < \ \frac{2\cdot 3^{n-1}}{n} \quad \ \text{for \ $n\geq 3$.}
$$

Finally, let us mention that the order polynomial \ts $\Omega(P,t)$ \ts can have
negative coefficients, implying that~\eqref{eq:OP-gen} does not follow directly
from the leading term of \. $t^n e(P)/n!$ \.  For example, recall that
$$
\Omega(P_5,t) \ = \
1^{4}\. + \. 2^{4} \. + \. \ldots \. +  \. t^{4} \ = \ \frac{1}{30}\.
\bigl(6 \ts t^5 \ts + \ts 15 \ts t^4 \ts + \ts 10 \ts t^3 \ts - \ts t\bigr),
$$
and note the negative coefficient in~$t$.
\end{rem}

\medskip

\section{Restricted linear extensions} \label{s:restricted}

In this section we use an algebraic approach to obtain vanishing and uniqueness
conditions for the generalized Stanley inequality.  We also present a direct
combinatorial argument for the uniqueness conditions.

\subsection{Background}\label{ss:restricted-background}
%
%
Before we proceed to generalizations, let us recall some definition
and results about group action on the set~$\Ec(P)$ of linear extensions.
In our presentation we follow Stanley's survey~\cite{Sta-promo}.

Let \ts $X=(X,\prec)$ \ts be a poset on \ts $|X|=n$ \ts elements.
\defn{Promotion} \. $\partial: \Ec(P) \to \Ec(P)$  \. is a bijection
on linear extensions defined as follows.  For \ts $f\in \Ec(P)$,
let \. $t_1 \prec \ldots \prec t_r$ \.
be a maximal chain in~$P$ such that $f(t_1),f(t_2),\ldots,f(t_r)$ is
lexicographically smallest. Define \. $f\partial\in \Ec(P)$ \. as
$$f\partial(x) =
\begin{cases} \ f(t_{i+1})-1 & \ \ \text{ if \ }x=t_i\text{ \, for some } \, i<r\ts, \\
\ n &  \ \ \text{ if \ $x=t_r$}\ts,\\
\ f(x)-1\ts &  \ \ \ \text{otherwise} .
\end{cases}
$$
We think of~$\partial$ as an operator applied on the right, and write \. $\partial: \ts f \ts\mapsto\ts f\ts \partial$.

\defn{Evacuation} \. $\ve: \Ec(P) \to \Ec(P)$  \. is a another operator
on linear extensions defined as follows.   Denote by \ts $\partial_i$ \ts
as the promotion on a poset obtained by restriction to elements with $f$-values \.
$1,\ldots,i$, so that \ts $\partial_n=\partial$ \ts and \ts $\partial_1=1$. Then \ts $\ve$ \ts is defined
as the composition \. $\ve := \ts\partial_n \ts \circ \ts\dots\ts\circ\ts\partial_1$\ts,
and we write \. $f\ve = f\ts\partial_n \ts \cdots\ts\partial_1$\ts.

\smallskip

The promotion and evacuation maps can be interpreted using group actions on linear extensions as follows.

\smallskip

Let \ts $\RG_n=\<\tau_1,\ldots,\tau_{n-1}\>$ \ts be an infinite Coxeter group with the relations
\begin{equation}\label{eq:Coxeter}
\tau_1^2\. = \. \ldots \. = \tau_{n-1}^2 \. = \. 1 \qquad \text{and} \qquad \tau_i\ts\tau_j\. = \. \tau_j\ts \tau_i
\ \ \text{ for all } \ \, |i-j|\. > \.1\ts.
\end{equation}
Note that the symmetric group \ts $S_n$ \ts is a quotient of~$\ts\RG_n$.
We also define elements \. $\de_2,\ldots,\de_{n}=\de\in \RG_n$ \. as follows:
$$\delta_k \. := \. \tau_1 \ts \tau_2 \.\cdots \.\tau_{k-1} \ \ \. \text{for} \ \ 1< k \le n\.,
\quad \text{and} \quad \gamma \. := \. \de_n \ts \de_{n-1} \ts \cdots \ts \de_2 \ts.
$$
Note that \. $\RG_n=\<\de_2,\ldots,\de_n\>$, and that \ts $\ga$ \ts is an involution: \ts $\ga^2=1$, see e.g.\
\cite[Lemma~2.2]{Sta-promo}.

With every linear extension \ts $f\in \Ec(P)$ \ts we associate a  word \. $\xx_f \ts=\ts x_1\ldots x_n\ts \in X^\ast$,
such that \. $f(x_i) =i$ \. for all \. $1\le i \le n$.  In the notation of the previous section, this says
that \. $X=\{x_1,\ldots,x_n\}$ \. is a natural labeling corresponding to~$f$.
%

We can now define the action of \ts $\RG_n$ \ts on \ts $\Ec(P)$ \ts as the right
action on the words \ts $\xx_f$, \ts $f\in \Ec(P)$.  For \. $\xx_f \ts =\ts x_1\ldots \ts x_n$ \. as above, let
\begin{equation}\label{eq:tau-def}
(x_1\ldots \ts x_n) \. \tau_i \ := \ \begin{cases} \ x_1 \ldots \ts x_n, & \ \text{if \, $x_i \prec x_{i+1}$}\ts,\\
\ x_1\dots x_{i+1} \ts x_i \dots x_n\ts, & \ \text{if \, $x_i \parallel x_{i+1}$}\ts.
\end{cases}
\end{equation}

Observe that if \. $1= i_1 < i_2 <\dots <i_r\le n$ \. are the indices of the lexicographically smallest maximal chain in the linear extension~$f$, then
$$
(\xx_f) \. \delta \, = \, (x_1\ldots \ts x_n) \. \delta \, = \, x_2 \ldots \ts x_{i_2-1} \ts x_{i_1}
\ldots \ts x_{i_r-1} \ts x_{i_{r-1}} \ldots \ts x_{i_r} \, = \, \xx_{f\partial}\.,
$$
where \ts $\de=\de_n$ \ts as above and $\xx_f \delta = \xx_{f\partial}$ is the promotion operator.

\smallskip

\begin{prop}[{\rm see e.g.~\cite[Prop.~4.1]{AKS}}{}]
\label{p:LE-transitive}
Let \. $P=(X,\prec)$ \. be a poset with \. $|X|=n$ \. elements.
Then group \. {\rm $\RG_n$} \ts acts transitively on \ts $\Ec(P)$.
\end{prop}

\smallskip

\begin{rem}
The proposition is a folklore result repeatedly rediscovered in
different contexts.  For the early proofs and connections to Markov chains,
see~\cite{KK,Mat}.  For a brief overview of generalizations and further
references, we refer to the discussion which follows Prop.~1.2
in~\cite{DK1}.
\end{rem}

\smallskip

\subsection{Generalization to restricted posets}\label{ss:restricted-gen}
Let \ts $P=(X,\prec)$.  Fix a sequence of $k$ elements \. $\uu=(u_1,\ldots,u_k) \in X$ \.
and a sequence of $k$ distinct integers \. $\aa = (a_1,\ldots,a_k)$, such that \.
$1\leq a_1 < \ldots < a_k \leq n$.  A \defn{restricted linear extension}
with respect to \ts $(\uu,\aa)$ \ts is a linear extension \ts $f\in \Ec(P)$ \ts
 such that \. $f(u_i)=a_i$ \. for all \. $1\le i\le k$. As in the introduction, we
 denote this set by \. $\Ec(P,\ts\uu,\ts\aa)$.

Our first goal is to modify and generalize Proposition~\ref{p:LE-transitive} for
restricted linear extensions.
For simplicity assume that \. $a_i +1 < a_{i+1}$ \. for all \. $1\le i< k$.
Otherwise, we can identify elements \ts $u_i\sim u_{i+1}$ \ts and consider
the equivalent problem for the so obtained smaller poset. We also assume $a_1>1$ and $a_k<n$, since otherwise the corresponding element $u_1$ or $u_n$ should be minimal/maximal and can be removed from the poset, again reducing the problem.
We also assume that $u_i \prec u_{i+1}$ in the poset.

\smallskip

Let \. $A:=\{a_1,\ldots, a_k\}$ \. and \. $A'=\{a_1-1,\ldots, a_k-1\}$, so \. $A\cap A'=\emp$ \.
by the assumption.
Let
$$
\RH_n(\aa) \, = \, \<\ts\tau_i, \. \si_r \ : \ 1\le i < n, \. i\notin A \cup A', \. 1\le r\le k\ts\>
$$
be an infinite group
with relations as in~\eqref{eq:Coxeter} and \. $\si_r^2=1$ \. for all \. $1\le r \le k$.
This is a free product of several infinite Coxeter groups which acts on \ts $\Ec(P, \ts \uu,\ts \aa)$ \ts
as follows.

First, for all \. $i\notin A \cup A'$, \. $1\le i < n$, the action of \ts $\tau_i$ \ts
defined in~\eqref{eq:tau-def} can be restricted  to act on \ts $\Ec(P, \ts \uu,\ts\aa)$.
Next, for all \. $1\le r \le k$ \. and \. $j=a_i$ \. define the action
of~$\ts \si_i$ \ts on \ts $\Ec(P,\ts \uu,\aa)$:
$$
(x_1\ldots \ts x_n) \. \sigma_i \ := \ \begin{cases}  \ x_1 \ldots \ts x_{j+1} \ts x_j \ts x_{j-1} \ldots \ts x_n &
\ \text{if } \ \ x_{j-1} \parallel  x_j\., \ \. x_{j-1} \parallel x_{j+1} \ \ \text{and} \ \, x_{j}  \parallel  x_{j+1}\,,\\ \
x_1\ldots \ts x_{j-1} \ts x_j \ts x_{j+1} \ldots \ts x_n & \text{otherwise}\ts.
\end{cases}
$$
Here we continue using our convention of association of \ts $\xx_f$ \ts with \. $f \in \Ec(P,\ts \uu,\aa)$.

Note that when \. $x_{j-1}  \parallel  x_j$ \. and \. $x_{j}  \parallel  x_{j+1}$ \. we have
 $$
 \xx \. \sigma_i \, = \xx \. \tau_j \ts \tau_{j-1} \ts \tau_j \, = \, \xx \. \tau_{j-1}\ts \tau_j\ts \tau_{j-1}\ts.
 $$
However, when \. $x_{j-1} \prec x_{j}$ \. and \. $x_j  \parallel  x_{j+1}$ \. we have \. $\xx \. \sigma_i \. = \. \xx$, but
$$\xx \. \tau_{j-1} \ts\tau_j\ts \tau_{j-1} \, \neq  \, \xx,
$$
since \. $x_j$ \ts has been moved to $(j+1)$-st position. The same property holds when
 \. $x_{j-1} \parallel  x_{j}$ \. and \. $x_j \prec x_{j+1}$\ts.

\smallskip

\begin{ex}\label{e:group-H-not}
Let us note that the action of \ts $\RH_n(\aa)$ \ts on \ts $\Ec(P, \ts\uu,\ts\aa)$ \ts
is not necessarily transitive. For example, let \ts $P=\bigl(X,\prec)$, where
\. $X=\{x,y,u_1,z,u_2\}$, be a poset isomorphic to \. $C_3 + C_1 + C_1$ \. with \.
$u_1 \prec z \prec u_2$ \. and $x,y$ incomparable to \. $\{u_1,z,u_2\}$.
Now, when \. $a_1=2$ \. and \. $a_2=4$, the action of group $\RH_5(\aa)$ \. has two orbits:
\. $\{xu_1zu_2y\}$ \. and \. $\{yu_1zu_2x\}$.  This shows that to generalize
Proposition~\ref{p:LE-transitive} we need to enlarge group~$\RH_5(\aa)$.
\end{ex}

\smallskip

Let \. $\GG_n=\bigl\<\tau_{i\ts j} \.:\. 1\leq i<j\leq n\bigr\>$ \. be an infinite group
with relations
\begin{equation}\label{eq:tau-G-def}
\aligned
 \tau_{i\ts j}^2 \. = \. 1  & \qquad \text{for all } \ \ 1\le i<j\le n\ts, \\
 \tau_{i\ts j} \. \tau_{k\ts \ell} \. = \. \tau_{k\ts  \ell} \.\tau_{i\ts j} & \qquad \text{for all } \ \  i<k<\ell<j\ \text{ or }\  i <j<k<\ell\ts.
\endaligned
\end{equation}
Define the action of \. $\GG_n$ \ts on \ts $\Ec(P)$ \ts as
$$(x_1\ldots \ts x_n) \. \tau_{i\ts j} \ := \ \begin{cases}
\ x_1\ldots \ts x_j \ldots \ts x_i \ldots x_n & \
\text{if \. $x_i \parallel y$ \. and \. $x_j \parallel y$ \. for all \. $y \in\{x_{i+1},\ldots,x_{j-1}\}$}, \\
\ x_1 \ldots \ts x_n & \ \text{otherwise}.
\end{cases}
$$
In the notation above, we have \. $\tau_{i \. i+1} = \tau_i$\., so \. $\RG_n \subset \GG_n$ \. is a subgroup.
For brevity, we write \ts $\tau_i$ \ts for \ts $\tau_{i \. i+1}$ \ts from this point on.

Finally, let \. $\GG_n(\aa)$ \. be a subgroup of \ts $\GG$ \ts defined as follows:
$$\GG_n(\aa) \, := \, \bigl\<\ts \tau_{i\ts j} \, : \, i,j\notin A, \,  1\le i<j\le n \ts\bigr\>\ts.
$$

\smallskip

\begin{thm}\label{t:LE-gen-transitive}
Let \. $P=(X,\prec)$ \. be a poset  with \. $|X|=n$ \. elements.
Fix a chain of $k$ elements \. $\uu=(u_1,\ldots,u_k) \in X$ \.
and an increasing sequence of $k$ distinct integers \. $\aa = (a_1,\ldots,a_k)$.
Then group \. {\rm $\GG_n(\aa)$} \ts defined above
acts transitively on \ts $\Ec(P, \ts \uu,\ts\aa)$.
\end{thm}

\begin{proof}
Suppose \. $\Ec(P,\ts \uu,\ts \aa) \neq \emp$. Fix \. $f \in \Ec(P,\ts \uu,\ts\aa)$ \. and write \.
$\xx_0=x_1\ldots x_n$ \. corresponding to the natural labeling \. $f(x_i)=i$.  For every \.
\. $\yy = y_1\ldots y_n$ \. corresponding to a linear extension \. $g \in \Ec(P,\ts \uu,\ts \aa)$,
let \ts $\inv(\yy)$ \ts be the number of inversions in the permutation~$f(\yy):=\bigl(f(y_1),\ldots,f(y_n)\bigr)$.  We claim that
unless \. $\yy=\xx_0$, we can use operators in \. $\GG_n(\aa)$ \. to decrease \ts $\inv(\yy)$.
Using this recursively, we can then reach \ts $\xx_0$ \ts as the unique element for which~$\ts \inv(\xx_0) =0$.

Consider the permutation \. $w:=\bigl(f(y_1),\ldots,f(y_n)\bigr)$.  If \. $w\ne \mathbf{1}$,
there exist elements \. $f(y_i), f(y_{i+1}) \in w$ \. such that \. $f(y_i)>f(y_{i+1})$.
Then we have \. $y_i  \parallel  y_{i+1}$.  We call such \. $(y_i, y_{i+1})$ \. a \defn{descending pair}.
Suppose there is a descending pair with \. $y_i, y_{i+1} \not \in \uu$.
Then \. $\tau_i \in \RH_n(\aa)$, and \. $\yy\. \tau_i \ts = \ts \ldots y_{i+1}\ts y_i\ldots$ \ts
Therefore, we have \. $\inv(\yy\ts\tau_i)\ts =\ts \inv(\yy) -1$, which proves the claim in this case.

In the remaining cases, every descending pair involves at least one element from~$\uu$,
which are fixed points of the labeling~$f$.  Suppose there are two adjacent descending pairs, i.e.\
\. $f(y_{i-1})>f(y_i)>f(y_{i-1})$ \. and \. $y_i \in \uu$. Then we have \.
$\inv(\yy\tau_{i-1 \. i+1})=\inv(\yy)-3$, which prove the claim in this case.

Finally, suppose every descending pair involves at least one element from~$\uu$, and none are adjacent.
Let  \ts $i_1-1$ \ts be the last descent of $w$. Then \ts $i_1-1\in \aa$, i.e. $i_1-1=a_t$ \. for some $a_t\in \aa$,
and we have \.
$a_t=w_{a_t} >w_{i_1}$.  To see this, suppose the contrary that the last descent in $w$ is at $i_1 -1 =a_t-1$, so $w_{i_1-1} > w_{a_t}=a_t$ and all elements of $w$ after $a_t$ are increasing. Since $a_t$ is a fixed point and we have $n-a_t$ positions after $a_t$ filled with numbers larger than $a_t$, i.e. from the interval $\{a_t+1,\ldots,n\}$, we must have $w_i=i$ for $i=a_t,\ldots,n$, so all values $>a_t$ appear after it. Thus $w_{a_t-1}<a_t =w_{a_t}$, reaching a contradiction, and so the last descent is at $a_t$.


Let \. $m <i_1-1$ \. be the largest value for which \. $w_{m}>i_1-1=a_t$. Such value exists
since by the reasoning above at least one of \. $\{i_1,\ldots,n\}$ \. appears before~$i_1-1$.  Now,
form a sequence \. $i_1\ts>\ts i_2\ts>\ts\ldots\ts>\ts i_r\ts >\ts i_0= m$\ts, such that \ts
$i_j$ \ts is the largest index smaller than \ts $i_{j-1}$ \ts  such that \. $w_{i_j}<w_{i_{j-1}}$.
Note that by similar interval arguments we must have that \. $f(y_{i_j})<i_j$ \.
and they do not hit any of the elements in~$\aa$.  Note that these indices give a maximal
increasing subsequence in~$w_{m+1}\ldots w_n$ which ends at~$w_{i_1}$.

Now apply \. $\tau_{m\. i_r}\. \tau_{i_r \. i_{r-1}} \. \cdots \. \tau_{i_2\. i_1}$, which is nontrivial as it transposes the element $y_m$, incomparable to all elements in positions $\in [m+1,i_1]$, with the elements $y_{i_r},\ldots$ which are also incomparable with the elements in the corresponding interval.
Note that this is the cycle permutation \. $(m,i_1,i_2,\ldots)$ \. so that $y_{m}$ moves to position $i_1$,
and the other elements slide down. This give a linear extension where the elements sliding to left bypass
only elements of larger value of~$f$, and hence respect the partial order.  Since \ts $f(y_{m})$ \. is
larger than the elements it jumps over, the resulting permutation has fewer inversions.  This proves the claim
in that case and completes the proof of the theorem.
\end{proof}

\smallskip

\subsection{Vanishing conditions}\label{ss:restricted-vanish}
For \. $\aa=(a_1,\ldots,a_k)$, let \. $\aa^{\<i\>} := (a_1,\ldots,a_i+1,\ldots,a_k)$.
By definition, the operator
$$
\tau_{a_i}\, : \ \Ec(P,\uu, \aa) \cup \Ec\bigl(P,\uu,\aa^{\<i\>}\bigr) \ \to \
\Ec(P,\uu,\aa) \cup \Ec\bigl(P,\uu,\aa^{\<i\>}\bigr)
$$
is an involution. For \. $i< j$, let
$$\delta_{i \ts j} \, := \, \tau_i \. \tau_{i+1} \. \cdots \. \tau_{j-1} \quad \text{and} \quad
\delta_{j\ts i} \, := \, \tau_{j-1} \. \cdots \. \tau_{i+1} \. \tau_i
$$
be the \defn{promotion operator} starting at position~$i$ and ending in position~$j$, and
the \defn{demotion operator} starting at position~$j$ and ending in position~$i$.

\smallskip

\begin{proof}[Proof of Theorem~\ref{t:vanish}]
Without loss of generality, we can assume that poset \. $P=(X,\prec)$ \.
has a unique minimal element~$\wh{0}$ and unique maximal element~$\wh{1}$.
Since \. $f\big(\wh{0}\big) = 1$ \. and \. $f\big(\wh{1}\big) = n$ \. for
every linear extension \. $f\in \Ec(P)$, we can also add \. $u_0 = \wh{0}$ \.
and \. $u_{k+1} = \wh{1}$ \. to the chain \.
$u_1\prec \ldots \prec u_k$\ts, and set \. $a_0=1$, \. $a_{k+1}=n$.
Equation~\eqref{eq:Sta-gen-vanish} then simplifies to
\begin{align}\label{eq:a_conditions}
a_j\. - \. a_i \ > \ h(u_i,u_j) \ \quad \text{ for all } \quad 0\leq i<j \leq k+1.
\end{align}

First, let us show  that inequalities~\eqref{eq:a_conditions} always hold.
Indeed, in every word \ts $\xx_f$ \ts
corresponding to a linear extension \. $f\in \Ec(P,\uu,\aa)$ \. with \. $f(u_i)=a_i$\ts,
we must have the elements from \. $\bigl(u_i,u_j\bigr)_P$ \. lie between~$u_i$ and~$u_j$,
and hence \. $a_j-a_i>h(u_i,u_j)$.

\smallskip

In the opposite direction, assume that the inequalities~\eqref{eq:a_conditions} hold
for all \. $0\leq i<j \leq k+1$.  To prove that \. $\Ec(P,\uu,\ts\aa)\ne \emp$,
proceed by induction on~$k$.
For \ts $k=1$, let \ts $\al$ \ts be a word obtained from totally ordering of the poset
interval \. $\big(\wh{0},u_1\big)$, and \ts $\be$ \ts be a word obtained from totally ordering
of the poset interval \. $\big(u_1,\wh{1}\big)$. Order the remaining elements of \ts
$P-u_1$ \ts into a word~$\ga$, and then insert \ts $u_1$ \ts at position \ts $a_1$ \ts in the
concatenation \. $\al\ts\ga\ts\be$. Since \ts $u_1 \parallel \ga$, \. $a_1>|\al|$, and \. $n-a_1>|\be|$,
this is a linear extension in \ts $\Ec(P,u_1,a_1)$.

Suppose now that the result holds for all sequences of length \. $k\ge 1$,
and let \. $\yy \in \Ec(P,\ts \uu,\ts \aa)$. Now let \ts $u_{k+1}$ \ts
be another element and \ts $a_{k+1}$ \ts satisfy the conditions
in the statement. Suppose that the position of \ts $u_{k+1}$ \ts is at \. $a'\neq a_{k+1}$.
Let us show that if \ts $a'<a_{k+1}$, then we can move \ts $u_{k+1}$ \ts to position \ts $a'+1$
without moving the other~$u$'s, and if \ts $a'>a_{k+1}$ we can move \ts $u_{k+1}$ \ts one position down.
Repeating this we will eventually get \ts $u_{k+1}$ \ts at a position \ts $a'=a_{k+1}$, to obtain
the desired linear extension.

\smallskip

From now on we act with the group $G$, and the promotion and demotion operators $\delta_{ij}$ to transform $\yy$.

Let \. $a'<a_{k+1}$. Since \. $n-a'>n-a_{k+1}\geq h(u_{k+1},\wh{1})$, there must be at least
one element in~$\yy$ appearing after \ts $u_{k+1}$ \ts which is incomparable to \ts $u_{k+1}$; denote by \ts $z$ \ts the first such element, at position $t$. Then the elements between \ts $u_{k+1}$ \ts and \ts $z$ \ts are incomparable with $z$, since they must be $\succ u_{k+1}$. Inserting~{\ts $z$} \ts immediately before \ts $u_{k+1}$ \ts, i.e.\ forming the word \ts $\yy\delta_{t a'}$, then respects the partial order and shifts \ts $u_{k+1}$ \ts to position~$\ts a'+1$.
Note that this transformation does not move $u_1,\ldots, u_k$ since we assume that $u_k \prec u_{k+1}$, which implies that $a_k <a$.

Suppose now that \. $a'>a_{k+1}$. Then \. $a'> h(\wh{0},u_{k+1}) +1$, so there is an element before \ts $u_{k+1}$ \ts
which is incomparable to \ts $u_{k+1}$. Let \ts $z_0$ \ts be the last such element, and suppose that it is at position \. $i_0\in (a_{r-1},a_{r})$. Note that the elements between \ts $z_0$ \ts and \ts $u_{k+1}$ in~$\yy$ must be incomparable to~{\ts $z_0$}, since by minimality they must be all~{\ts$\prec u_{k+1}$}. If $z_0$ appears after~$\ts u_k$, then we obtain $\yy\delta_{i_0a'}$, where $z$ is inserted  after \ts $u_{k+1}$ \ts and  \ts $u_{k+1}$ shifts to position \ts $a'-1$.  Otherwise, since \. $a'-a_k > h(u_{k+1},u_k)+1$, there is an element \ts $z_k$ \ts between $u_k$ and $u_{k+1}$ in $\yy$ such that  $z_k$ is incomparable to either $u_k$ or $u_{k+1}$. Since \. $z_k \neq z_0$, we must have \. $z_k  \parallel  u_k$, and let this \ts $z_k$ \ts be the first such element after \ts $u_k$ \ts at position~$i_k$.

In general, for every \.$t \in [r,k]$\. we define $z_t$ at position \.$i_t<a'$\. to be the first element after \.$u_t$\., such that \ts $z_t \parallel u_t$. Note that such element exists, which is seen as follows. Since  \.$h(u_t,u_{k+1})+1<a'-a_t$\. there is an element \.$z$ \. between \.$u_t$ \. and \. $u_{k+1}$ \. incomparable to at least one of them. However, for all such \.$z\prec u_{k+1}$\., so we must have \.$z \parallel u_t$. Next, observe that \.$z_t$  is incomparable to all elements in $\yy$ appearing between $u_t$ and $z_t$.
Now transform $\yy$ as follows. First, let \ts $\yy^0:= \yy \delta_{i_0 \. a'}$, and note that here $z_0$ is sent to position $a'$ and all elements in between have been shifted down one position. Let \ts $a'_t:=a_t-1$ \ts and \ts $i'_t:=i_t-1$ \ts be the positions of $u_t$ and $z_t$ in~$\yy^0$. Next, let $\yy^1 = \yy^0 \. \delta_{i'_r \. a_r'}$ which moves $z_r$ before $u_r$, so the position of $u_r$ is restored to $a_r$ as well as all elements between them. Suppose that $i_r \in (a_{p-1}, a_{p})$. Let then $\yy^2:=\yy^1\. \delta_{i'_p \. a'_p}$, so the element $z_p$ is demoted to the position before $u_p$, and thus all other elements at positions $[a_p,\.i_p]$ have now restored their original position from~$\yy$. Continuing this way, if \. $i_p \in (a_{q-1},\.a_q)$\. we obtain $\yy^3:= \yy^2\. \delta_{i'_q\. a'_q}$ and so on until we have shifted
all elements \ts $u_r,\ldots,u_k$ \ts to their positions \ts $a_r,\ldots,a_k$. Also, \ts $u_{k+1}$ \ts is at position~$a'-1$, which is what we needed to show. This completes the proof.
%
\end{proof}

\smallskip

\begin{proof}[Proof of Corollary~\ref{cor:Sta-gen-vanish-poly}]
The first part follows trivially from~\eqref{eq:Sta-gen-vanish} or, equivalently,
its simplified version~\eqref{eq:a_conditions}.  For the second part, note that the proof above is
completely constructive and builds \. $f\in \Ec(P,\uu,\aa)$ \. in polynomial time starting with
a linear extension \. $g\in \Ec(P)$.  The details are straightforward.
\end{proof}

\smallskip

\subsection{Uniqueness conditions}\label{ss:restricted-unique}

In the next lemma, we show that, given a linear extension \ts $f \in \Ec(P,\uu,\aa)$,
we can check if such~$f$ is unique in polynomial time.

In the notation of  Theorem~\ref{t:vanish},
let \. $v_i:=f^{-1}(a_i-1)$ \. and \. $w_i:=f^{-1}(a_i+1)$ \.
for \. $1\le i \le k$.
We adopt the convention that \ts $v_1=\widehat{0}$ \ts if \ts $a_1=1$,
and \ts $w_k=\widehat{1}$ \ts if \ts $a_k=n$.  For \. $1 \leq i \leq  j \leq n$,
let
\[ f^{-1}[i,j] \ := \ \bigl\{\ts f^{-1}(i)\., \.\ldots \.,  f^{-1}(j)\ts  \bigr\}.
\]

\smallskip

\begin{thm}\label{t:gen-Stanley-unique}
In the notation of Theorem~\ref{t:vanish}, let \. $f \in \Ec(P,\aa,\uu)$ \.
be a linear extension as in the theorem.
Then \ts $|\Ec(P,\uu,\aa)|=1$ \, \underline{if and only if} \,
	the following conditions hold:
	\begin{equation*}
\aligned
		& (1) \quad \text{$f^{-1}[a_i+1,a_{i+1}-1]$ \. forms a chain in $P$  for every \. $1\le i \le k$, \ts and } \\
	& (2) \quad  \text{There are no \. $1 \leq i \leq j \leq k$\ts, \. such that \. $\{v_i,w_j\} \ts \parallel \ts f^{-1}[a_i,a_j]$.}	
\endaligned
\end{equation*}
\end{thm}


\begin{proof}
For the \ts $\Rightarrow$ \ts direction, note that~$(1)$ follows directly. Indeed,
recall that \. $e(P)>1$ \.  unless \ts $P$ \ts is a chain.  Therefore if the restriction of \ts $P$ to
\. $\big\{f^{-1}(a_i+1),\ldots,f^{-1}(a_{i+1}-1)\big\}$ \. is not a chain, then there is more than one
linear extension over these elements, which extends to the desired linear extension \. $g\in \Ec(P,\uu,\aa)$, \. $g\ne f$.
For~$(2)$, suppose to the contrary that \. $v_i,w_j \ts \parallel \ts f^{-1}[a_i,a_j]$.
Then swapping the value of \ts $f(v_i)$ \ts and $f(w_j)$ \ts via $\xx_f \tau_{a_i-1 \ts a_j+1}$ we  get a new linear extension,
a contradiction.
	
\smallskip

For the \ts $\Leftarrow$ \ts direction, suppose that~(1) and~(2) hold and there exists \ts $g\in \Ec(P,\uu,\aa)$ \ts
for some \ts $g\ne f$. By Theorem~\ref{t:LE-gen-transitive} we have that there is an element \ts $\pi \in \GG_n(\aa)$, s.t. $\xx_f \pi = \xx_g$. Write $\pi$ as the minimal (reduced) product of transpositions $\pi = \tau_{r_1 \ts s_1} \cdots$ which act nontrivially. So $\xx_f \tau_{r_1 \ts s_1} \neq \xx_f$ and thus $\{f^{-1}(r_1), f^{-1}(s_1)\} \ts \parallel \ts f^{-1}[r_1+1,s_1-1]$. Since the elements
in \ts  $f^{-1}[a_i+1,a_{i+1}-1]$ \ts form a chain we must have that $r_1 =a_i-1$ for some $i$ and that $s_1 = a_j+1$ for some $j$. Note that by definition \ts $r_1<s_1$ \ts and \ts $r_1,s_1 \not\in \aa$. Thus \. $\{ f^{-1}(a_i-1), f^{-1}(a_j+1)\} \parallel f^{-1}[a_i,a_j]$, and so condition~(2) does not hold, a contradiction.
\end{proof}

\smallskip

\begin{proof}[Proof of Corollary~\ref{cor:Sta-gen-uniqueness-poly}]
By the first part of Corollary~\ref{cor:Sta-gen-vanish-poly}, we can decide if
\. $|\Ec(P,\uu,\aa)|>0$ \. in polynomial time.  By the second part of the same
corollary, we can find a linear extension \. $f\in \Ec(P,\uu,\aa)$ \. in polynomial
time.  By Theorem~\ref{t:gen-Stanley-unique}, we can decide if such~$f$ is unique
in polynomial time.
\end{proof}

\medskip

\section{Injective proof of the Sidorenko inequality} \label{s:sid}

\subsection{Preliminaries} \label{ss:sid-pre}

Let $P=(X,\prec)$ be a poset with $|X|=n$ elements.  Denote by \ts $P|_J$ \ts the
restriction of \ts $P$ \ts to a subset \ts $J\ssu X$.  We write \ts $P-y$ \ts
to denote the restriction \ts $P|_{X-y}$.
Denote by  \. $P^\ast=(X,\prec^\ast)$ \. the \defn{dual poset}~:
$$x\prec^\ast y \quad \Longleftrightarrow \quad y \prec x\,, \quad \text{for all} \ \ x, \ts y \in X.
$$
Clearly, \ts $e(P^\ast)=e(P)$.

Denote by \ts $\cC(P)$ \ts the set of chains, and by $\cA(P)$ the set of
antichains in~$P$.   The \defn{comparability graph} \ts $\Com(P)=(X,E)$ \ts
is defined by \. $E=\bigl\{(x,y)~:~x\prec y, \. \text{ where } \. x,y\in X\bigr\}$. Note that
the chains in~$P$ are \emph{cliques} (complete subgraphs) in $\Com(P)$.
Similarly, the antichains in~$P$ are \emph{stable} (independent) \emph{sets} in $\Com(P)$.

Throughout this section, we think of the \defng{promotion} in a different way,
as a map from linear extensions to chains in the poset.
Formally, for $f\in \LE(P)$, let \. $x_1=f^{-1}(1)$. For \ts $i>1$, let $x_{i}\in X$ be
an element with the smallest value of~$f$ on \. $\{y~:~x_{i-1}\prec y\}$.
This gives a \defn{promotion chain} \. $C= \bigl[x_1\to x_2\to\ldots \to x_\ell\bigr]\in \cC(P)$, which
can also be viewed as the \emph{DFS path} in the  \emph{Hasse diagram} of~$P$.
Denote by \. $\Phi: \LE(P)\to \cC(P)$ \. the map \. $\Phi(f) = C$.

\bigskip

\begin{lemma}  \label{lem:DFS}
For all $P=(X,\prec)$  and $y\in X$ we have:
$$e(P-y) \. = \. \bigl|\bigl\{g\in \LE(P)\,:\, y \in \Phi(g)\bigr\}\bigr|
$$
\end{lemma}

\medskip

\begin{proof}
Consider a bijection
$$\vp: \,
\LE(P-y)
\.  \longrightarrow  \.
\bigl\{g\in \LE(P)\,:\,y \in \Phi(f)\bigr\}
$$
defined as follows.  Let \. $f\in \LE(P-y)$, and let
\. $[y\to x_1 \to \ldots \to x_k]$ \. be the
promotion path in the upper order ideal \. $B(y) = \{x\in X\.:\. x \succ y\}$.
Define \. $g=\vp(f) \in \LE(P)$ \. as follows.  Let \. $g(y):=f(x_1)$, \.
$g(x_i):=f(x_{i+1})$ \. for \. $1\le i < k$, and \. $g(x_k):=n$.  Observe
that in \ts $P^\ast$ \ts we now have \.
$\Phi(f) \. = \. \bigl[x_k\to\ldots \to x_1 \to y \to \ldots \bigr]$.
Reversing the role of \ts $P$ \ts and \ts $P^\ast$ \ts implies the result.
\end{proof}

\bigskip

\begin{cor}[{\rm {see~\cite{EHS}}}{}] \label{cor:antichain}
For every antichain $A\in \cA(P)$ we have
\begin{equation}\label{eq:EHS-ineq}
\sum_{y \in A} \. e(P-y) \, \le \, e(P).
\end{equation}
Furthermore, when \ts $A\subseteq X$ \ts is the set of minimal elements,
the inequality~\eqref{eq:EHS-ineq} is an equality.
\end{cor}

\begin{proof}
Note that for every \ts $C\in \cC(P)$ \ts and \ts $A\in \cA(P)$, we have \ts $|A\cap C|\le 1$.
Thus, we have:
$$\sum_{y \in A} \. e(P-y) \, = \,
\bigl|\bigl\{f\in \LE(P)\,:\,|\Phi(f)\cap A|=1\bigr\}\bigr| \, \le \, |\LE(P)| \, = \, e(P),
$$
which proves~\eqref{eq:EHS-ineq}.  For the second part, note that for every \. $f\in \Ec(P)$,
the promotion path \ts $\Phi(f)$ \ts starts with the minimal element in~$A$.  This implies
that~\eqref{eq:EHS-ineq} is an equality, as desired.
\end{proof}

\smallskip

\begin{rem}\label{r:comp}
Lemma~\ref{lem:DFS} is implicit in~\cite{EHS}, which only discusses
equality cases (cf.\ Corollary~\ref{cor:antichain}). By~\cite[Thm~4]{Sid},
map~$\Phi$ gives the \emph{linear extension flow} through $\Com(P)$ viewed as
directed network.  Although Sidorenko gives a combinatorial construction
of this flow in~\cite[Rem~1.2]{Sid}, this construction is also inexplicit.

Let us mention that second part of Corollary~\ref{cor:antichain} implies
by induction that \ts $e(P)$ \ts depends only on the comparability graph $\Com(P)$,
see~\cite{EHS,Sta-promo}.  The same holds for the order polynomial \ts $\Omega(P,t)$,
and can be proved using Ehrhart polynomials~\cite{Sta-two},
cf.~$\S$\ref{ss:finrem-Ehrhart}.  Alternatively, this result can be
shown via certain ``turning upside-down'' flips discussed in
\cite[Exc.~3.163]{Sta-EC}.
\end{rem}

\smallskip

\subsection{Sidorenko's inequality} \label{ss:sid-proof}
As in the introduction, let \. $P=(X,\prec)$ \. and \. $Q=(X,\prec')$ \.
be two posets on the same ground set, such that \. $|C\cap C'|\le 1$ \.
for all \. $C\in \cC(P)$ \. and \. $C'\in \cC(Q)$.
Then \. $\cC(P) \subseteq \cA(Q)$ \. and \. $\cA(P) \subseteq \cC(Q)$, by definition.
%

\bigskip

\begin{lemma}[{\rm cf.~\cite[Lemma~10]{Sid}}{}] \label{lem:main}
For all $P$ and $Q$ as above, we have:
$$
\sum_{y\in X} \, e(P-y) \. e(Q-y) \, \le \, e(P) \. e(Q)\ts.
$$
\end{lemma}

\smallskip

\begin{proof}  We have:
$$
\aligned
\sum_{y\in X} \. e(P-y) \. e(Q-y) \, & =_{\text{Lem~\ref{lem:DFS}}} \,
\sum_{y\in X} \. e(Q-y) \, \sum_{C\in \cC(P) \ : \ C\ni y} \. |\{f\in \LE(P)\,:\,\Phi(f)=C\}|
\\& = \hskip.95cm
\sum_{C\in \cC(P)} \, \sum_{y\in C} \. e(Q-y) \, \cdot \, |\{f\in \LE(P)\,:\,\Phi(f)=C\}|
\\
& \le_{\text{Cor.~\ref{cor:antichain}}}
\,
\sum_{C\in \cC(P)} \, e(Q) \, \cdot \, |\{f\in \LE(P)\,:\,\Phi(f)=C\bigr\}|
\\
& \le \   e(Q)  \,  \sum_{C\in \cC(P)} \,|\{f\in \LE(P)\,:\,\Phi(f)=C\}| \quad = \
e(P) \. e(Q).
\endaligned
$$
Here in the third line, Cor.~\ref{cor:antichain} applies to poset~$Q$, since every chain in~$P$
is an antichain in~$Q$.
\end{proof}

\smallskip

\begin{proof}[Proof of Theorem~\ref{t:sid}]
The theorem follows from Lemma~\ref{lem:main}, by induction on $n=|X|$.
\end{proof}

\smallskip

\begin{cor}[{\cite[Thm~11]{Sid}}{}]\label{c:sid-equality}
In notation of Theorem~\ref{t:sid}, the inequality~\eqref{eq:sid} is
an equality \. \underline{if and only if} \. $P$ \ts
is a \emph{series-parallel poset}.
\end{cor}

\smallskip

The result is well-known and follows easily by tracing back the inequalities in
the proof of Lemma~\ref{lem:main}.  We omit the details.

\smallskip

\begin{proof}[Proof of Theorem~\ref{t:sid-gen}]
We can rewrite the proof of Lemma~\ref{lem:main} as follows:
$$
\aligned
\sum_{y\in X} \. e(P-y) \. e(Q-y)
\ = \ \sum_{f\in \LE(P)} \. \sum_{g\in \LE(Q)} \. \sum_{y\in \Phi(P)\cap \Phi(Q)} \. 1 \
\le \  k \. e(P) \. e(Q).
\endaligned
$$
The result now follows by induction on~$n\ge k$, with the base \ts $n=k$ \ts trivial.
\end{proof}

\smallskip

\begin{proof}[Proof of Theorem~\ref{t:sid-SP}]
Let \. $\be: S_n \to \LE(P_\si)\times \LE\bigl(P_{\ov{\si}}\bigr)$ \. be
the injection defined implicitly by the proof of Theorem~\ref{t:sid} above.
First, observe that \ts $\be$ \ts is computable in polynomial time.  Indeed, by
induction, it is a composition of maps \ts $\be_i$ \ts each consisting of applying
maps \ts $\Phi$ \ts to posets
corresponding to partial permutations $\si_i:=(\si(1),\ldots,\si(i))$
and its dual \ts $\ov{\si_i}$, see the proof of Lemma~\ref{lem:DFS}.

Second, whenever defined, the inverse map \ts $\be^{-1}$ \ts can be computed by the proof
of Lemma~\ref{lem:DFS}, since the inverse of \ts $\Phi$ \ts on~$P$ is a map \ts $\Phi$ \ts on~$P^\ast$.
On the other hand, at each stage, the decision if the inverse of \ts $\be_i$ \ts exists
reduces to a problem whether a given antichain in the \. $Q_i:=P_{\ov{\si_i}}$ \. is a \defn{cut},
i.e.\ it intersects every chain in~$Q_i$.  This is a special case of directed graph
connectivity problem, and thus in~$\poly$. Putting this together implies that
we can decide in polynomial time if \. $(f,g) \in \be(S_n)$,
for all \. $f\in \LE(P_\si)$ \. and \. $g\in \LE\bigl(P_{\ov{\si}}\bigr)$.

In summary, the function \ts $\eta(\si)$ \ts counts the number of pairs of linear
extensions \ts $(f,g)$ \ts as above, such that \ts $(f,g) \notin \be(S_n)$.
Since the problem whether \ts $(f,g) \in \be(S_n)$ \ts can be decided in
polynomial time, this completes the proof.
\end{proof}

\medskip

\section{Final remarks and open problems}\label{s:finrem}

\smallskip

\subsection{Bj\"orner--Wachs inequality}\label{ss:finrem-HP}
In total, we include three proofs of the Bj\"orner--Wachs inequality:
the original injective proof in~$\S$\ref{s:injective}, the probabilistic proof
via Shepp's inequality in~$\S$\ref{s:OP-induction}, and Reiner's proof
via $q$-analogue in~$\S$\ref{s:q-analogue}.  Another proof was given
by Hammett and Pittel in~\cite[Cor.~2]{HP}, who seemed unaware of the
origin of the problem despite having~\cite{BW89} among the references.
Although somewhat lengthy and technical, their proof is completely
self-contained and is based on a geometric probability argument.
It is similar in spirit to Reiner's proof, but without benefits
of the brevity.

\smallskip

\subsection{Order polynomial} \label{ss:finrem-OP}
There is surprisingly little literature on the order polynomials
given that they emerge naturally in both P-partition theory and
discrete geometry.  We refer to~\cite{Joc} for order polynomials in
the case of symmetric posets, which are of independent interest,
and to~\cite{LT} for some computations.

It seems, there are more conjectures and open problems than results
in the subject.  It addition to the Kahn--Saks Conjecture~\ref{conj:KS-mon}, we have our own
Conjecture~\ref{conj:KS-FKG}.  We should warn the reader that there seem
to be insufficient effort towards testing of these conjectures, so it would
be interesting to obtain more computational evidence.

\smallskip

\subsection{Ehrhart polynomial}\label{ss:finrem-Ehrhart}
It is a classical observation by Stanley \cite{Sta-two},
that the order polynomial \ts $\Om(P,t+1)$ \ts is the
\defng{Ehrhart polynomial} of the corresponding
order polytope \ts $\mathcal{O}_P$\ts:
$$
\Ehr(\mathcal{O}_P,t) \, = \, \Om(P,t+1)\ts.
$$
This allows one to translate the results from combinatorial
to geometric language ad vice versa.

Notably, our Example~\ref{e:OP-Stanley} is motivated by Stanley's
\ts {\tt MathOverflow} \ts observation\footnote{Richard P.~Stanley,
\href{https://mathoverflow.net/q/200574}{mathoverflow.net/q/200574}
(March~20, 2015).} that the order
polynomial \ts $\Om(C_1 \oplus A_m,t+1)$ \ts can have negative
coefficients for \ts $m\ge 20$.
We refer to~\cite{LT} for more
on this example and to~\cite{Liu} for the background on non-negative
Ehrhart polynomials and further references.  We refer to
\cite{Cha} for $q$-Ehrhart polynomials, and to~\cite{KS} for further
results.

\smallskip

\subsection{Geometric form of the Kahn--Saks conjecture}\label{ss:finrem-KS}
One can ask if a version of the Kahn--Saks Conjecture~\ref{conj:KS-mon}
holds for general integral polytopes:
\begin{equation}\label{eq:KS-Ehr}
\text{Is} \quad \Ehr(Q,t-1)/t^d \quad \text{weakly decreasing for all} \ \ Q\in \rr^d \ \ \text{and} \ \ t \in \nn_{\ge 1} \. ?
\end{equation}
First, recall the example of
\defng{Reeve's tetrahedron} with vertices
at
$$
(0,0,0) \,, \quad (1,0,0) \,, \quad (0,1,0) \quad \text{and} \quad (1,1,h),
$$
see e.g.\ \cite[Ex.~3.23]{BR} and \cite[$\S$4.1]{GW}.  In this case,
the Ehrhart polynomial has negative signs, and the scaled Ehrhart polynomial
is non-monotone for large values of~$h$.  This shows that the
\defng{geometric Kahn--Saks conjecture}~\eqref{eq:KS-Ehr} does not hold for general lattice polytopes.

On the other hand, it is rather plausible that~\eqref{eq:KS-Ehr} holds for
\defng{antiblocking} (\defng{corner}) \defng{polytopes} \ts
(see e.g.~\cite[$\S$5.9]{Sch}) with integer vertices.
If true, this would imply the Kahn--Saks Conjecture~\ref{conj:KS-mon}.  Indeed,
although the order polytope $\ts \mathcal{O}_P$ \ts is not antiblocking, the
\defng{stable set} (\defng{chain}) \defng{polytope} \ts $\mathcal{C}_P$ \ts is both
altiblocking and has the same Ehrhart polynomial by Stanley's theorem:  \ts
$\Ehr(\mathcal{C}_P,t) = \Ehr(\mathcal{O}_P,t)$, see~\cite{Sta-two}.

Finally, let us
mention that the proof of Proposition~\ref{prop:scaled} can be modified to show that
$$\frac{1}{t^d} \. \Ehr(Q,t-1) \, \ge \, \frac{1}{(k \ts t)^d} \.\Ehr(Q,kt-1)\ts,
$$
for all antiblocking polytopes \ts $Q\in \rr^d$ \ts with integer vertices.
This gives some credence to our speculation~\eqref{eq:KS-Ehr} in this case.

\smallskip

\subsection{Log-concavity and $q$-log-concavity}\label{ss:finrem-log}
The log-concavity for order polynomials proved in Theorem~\ref{thm:log-concave}
is somewhat different from other log-concave inequalities, see e.g.\ \cite{Bra,Huh,CP,Sta2}.
The $q$-log-concavity in Corollary~\ref{c:q-log-concavity} is also classical albeit
less studied, see e.g.\ \cite{Kra,Ler,Sag}.

\smallskip

\subsection{Graham's conjecture}\label{ss:finrem-Graham}
We learned that of Daykin--Daykin--Paterson paper \cite[Thm~2]{DDP}
proving Graham's conjecture (Theorem~\ref{conj:Graham}) by accident,
while revising the paper.   We chose to keep our
Corollary~\ref{t:Graham-asy} as a nice application of our tools.  Most recently,
the first and second authors found a new proof of Theorem~\ref{conj:Graham} based
on the \emph{Ahlswede--Daykin inequality}, and further generalized this inequality
to a multivariate version \cite[$\S$9]{CP3}.

\smallskip

\subsection{Sidorenko inequality}\label{ss:finrem-sid}
Note that another combinatorial proof of Sidorenko's inequality (Theorem~\ref{t:sid})
was independently found in \cite[$\S$4.1]{GG}, where the authors
gave an elegant explicit construction of a surjection proving~\eqref{eq:sid}.  Unfortunately,
the proof of correctness of that surjection is technical and cannot be easily inverted
to obtain the desired injection.
More precisely, the authors give a explicit surjection \.
$\al: \LE(P_\si)\times \LE\bigl(P_{\ov{\si}}\bigr) \to S_n$\ts.
Unfortunately, the proof in~\cite{GG} is technical and indirect,
so an explicit injection requires further effort.

As we mentioned in the introduction, our injection~$\be$ defined implicitly
in the proof of Theorem~\ref{t:sid} likely coincides with an explicit injection
in~\cite{GG1}, since both essentially reverse engineer and make effective
the original proof by Sidorenko~\cite{Sid}.  The connection with the argument
in~\cite{StR} and the surjection in~\cite{MPP} in the case of \emph{Fibonacci posets}
remains unclear.

We also conjecture that the function \. $u: S_n \to \nn$ \.
defined by~\eqref{eq:Sid-function} \ts is \. $\SP$-complete.
The conjecture would follow if $\SP$-completeness was proved for self-dual
$2$-dimensional posets $P\simeq \ov{P}$. Unfortunately, the construction
in~\cite{DP} is too specialized and technical to obtain this result.

Finally, there a $q$-analogue of Sidorenko's inequality in~\cite[Cor.~3]{GG}
generalizing $q$-equality for the series-parallel posets given
in~\cite{Wei}.  See also~\cite{KS} for the definition of $e_q(P)$ for
general~$P$ based on the $P$-partition theory, and~\cite{BW} for many
other results on~$e_q(P)$.

\smallskip

\subsection{Mixed Sidorenko inequality}\label{ss:finrem-mixed-sid}
In~\cite{BBS}, Bollob\'as, Brightwell and Sidorenko showed how to obtain Sidorenko's
Theorem~\ref{t:sid} via a known special case of \defng{Mahler's Conjecture}.
Most recently, Artstein-Avidan, Sadovsky and Sanyal extended this approach
in~\cite{AASS} to obtain the following remarkable generalization of the
Sidorenko inequality.

For two posets \. $P=(X,\prec)$ \. and \. $Q=(X,\prec')$ \. on the same set,
\defn{mixed linear extensions} are triples \ts $(f,g,J)$, where \.
$J\ssu [n]$, \. $f\in \LE(P|_J)$, and \. $g\in \LE\bigl(Q|_{\ov{J}}\bigr)$. Denote
by \ts $e_k(P,Q)$ \ts the number of such triples with $|J|=k$, i.e.\
$$
e_k(P,Q) \ := \ \sum_{J\in \binom{[n]}{k}} \. e(P|_J) \. e\bigl(Q|_{\ov{J}}\bigr).
$$

\begin{thm}[{\rm \cite[Thm~6.2]{AASS}}{}] \label{t:sid-mxed}
Let \. $P, Q, S, T$ \. be four posets on the same ground set, such that \. $|C\cap C'|\le 1$ \. and \. $|D\cap D'|\le 1$,
for all \. $C\in \cC(P)$, \. $C'\in \cC(Q)$,  \. $D\in \cC(S)$ \. and \. $D'\in \cC(T)$.
Then we have:
\begin{equation}\label{eq:sid-mixed}
e_k(P,Q) \, e_k(S,T) \ \ge \ n!\.\ts \binom{n}{k}\ts.
\end{equation}
\end{thm}

It would be interesting to find a combinatorial proof of this result.
It would be even more interesting to find a direct injective proof, and
conclude that the function giving the difference of the two sides of~\eqref{eq:sid-mixed}
is in~$\ts \SP$.  The results in~\cite{IP} suggest that this might not be possible.
Finally, does the \defn{mixed Sidorenko inequality}~\eqref{eq:sid-mixed} have an upper bound
similar to that in~\cite{BBS}?

\smallskip

\subsection{Complexity of correlation inequalities}\label{ss:finrem-more}
By taking the limit \. $t\to \infty$ \. in Lemma~\ref{l:main order polynomial}, we obtain:
\begin{equation}\label{eq:min-min} \tag{$\ast$}
(n-1) \. \cdot \. e(P) \. \cdot \. e\bigl(P \sm \{x,y\}\bigr)  \ \geq \ n \. \cdot \. e\big(P \sm x\big)  \. \cdot \.  e\big(P\sm y\big).
\end{equation}
It would be interesting to see if this inequality can be proved injectively.
Is the function giving the difference of the two sides of this inequality in~$\ts \SP$?

Note that, by applying the negative-correlation version of the
FKG inequality to the proof of Lemma~\ref{l:main order polynomial},
we obtain the following result:

\begin{lemma}\label{l:OP min-max}
	Let \. $P=(X,\prec)$ \. be a poset, let \. $x\in X$ \. be a minimal element, and let \. $y \in X$ \. be a maximal element such that $y$ does not cover~$x$.
	Then, for every integer $t>0$, we have:
	\[ {\Omega\big(P,\ts t\big)} \. \cdot \. {\Omega\big(P \sm \{x,y\},\ts t\big)}  \ \leq \  {\Omega\big(P \sm x, \ts t\big)}  \. \cdot \.  {\Omega\big(P\sm y,\ts t\big)}.
	\]
\end{lemma}

By taking the limit \. $t\to \infty$ \. in the lemma, we get the inequality opposite to~\eqref{eq:min-min}:
\begin{equation}\label{eq:min-max} \tag{$\ast\ast$}
(n-1) \. \cdot \. e(P) \. \cdot \. e\bigl(P \sm \{x,y\}\bigr)  \ \leq \ n \. \cdot \. e\big(P \sm x\big)  \. \cdot \.  e\big(P\sm y\big).
\end{equation}
Of course, the element $y$ was minimal in~\eqref{eq:min-min} and is maximal in~\eqref{eq:min-max},
but these inequalities are striking in appearance.  Again, it would be interesting to see if this
inequality can be proved injectively.

\smallskip

\subsection{Vanishing and uniqueness conditions}\label{ss:finrem-vanishing}
Note that the vanishing conditions for the Stanley inequality are a
special case of the \defna{equality conditions}, which are fully described
in~\cite{SvH} and reproved in~\cite{CP}.  For example, Corollary~\ref{cor:antichain}
and Corollary~\ref{c:sid-equality} give further examples of equality conditions,
with a simple proof in both cases via direct injection.
%
%
When there is no injective proof, the equality condition can become a major challenge.  
On the other hand, the vanishing and uniqueness conditions tend to be much 
easier to establish using either combinatorial or geometric tools (see~\cite{EG}).

For example, for the \defng{Kahn--Saks inequality} generalizing Stanley's inequality, the
equality conditions remain open in full generality.  See, however, \cite[$\S$8]{CPP2}
for the vanishing conditions of the Kahn--Saks inequality, proved also via the promotion
technology.  See also \cite{CPP1,CPP2},
for the equality conditions of the Kahn--Saks and cross-product inequalities
for posets of width two.  Finally, let us mention Lemma~14.6 in~\cite{CP}, which is
yet another variation on Theorem~\ref{t:LE-gen-transitive} and proved by a direct
combinatorial argument using promotions.

In a different direction, sometimes the equality conditions are trivial as the
natural inequalities are always strict except for some degenerate cases.
This is the case with the \defng{XYZ inequality}~\cite{Fish}, and the
\defng{log-concavity} (Theorem~\ref{thm:log-concave-strict}) discussed above.

The uniqueness conditions are studied less frequently than vanishing and equality
conditions, since they tend to be harder.  For example, there is no description of
the uniquely colorable graphs, and this remains a major open problem~\cite{CZ}.
Notable positive results include uniqueness conditions for the \defng{Kostka
numbers}~\cite{BZ} and for the \defng{Littlewood--Richardson coefficients}
\cite[Prop.~3.13]{BI}.

\smallskip

\subsection{Poset dynamics}\label{ss:finrem-algebraic}
Promotions, demotions and evacuations were defined by Sch\"utzenberger in~\cite{Sch},
and this approach has been immensely influential leading to the \defng{RSK Algorithm} \ts
and the \defng{Edelman--Greene bijection}, among other things.
Group theoretic approach in the context of combinatorics of words were
developed by Lascoux and Sch\"utzenberger, and specifically in the generality
of posets were introduced by Haiman~\cite{Hai} and Malvenuto--Reutenauer~\cite{MR}.
See also~\cite{KB} for a related approach in the context of semistandard Young
tableaux, and~\cite{Sta-promo} for an extensive survey.

It would be interesting to find a generalization of the evacuation~$\ve$ which would
preserve its involution property.  Such
``restricted evacuation'' might give rise to ``restricted domino linear
extensions'' which would be of independent interest,
see e.g.~\cite[$\S$3]{Sta-promo}.

The extension of promotion to general bijections \. $X\to [n]$ \. was obtained
in~\cite{DK}.  Can our group action on restricted linear extensions be
generalized in this direction?  Note that we have only limited
understanding if the group action can be applied to study the
order polynomial.  See however~\cite{Hop} for some elegant
product formulas in some special cases.

The promotion operators were used in~\cite{AKS} to define a Markov chain on
the set \ts $\Ec(P)$ \ts of linear extensions of a given poset~$P$.
See also~\cite{RS} where a related Markov chain was shown to be
mixing in time \ts $O(n \log n)$.  It would be interesting to see if these
results can be generalized to show that the restricted linear extensions in
\. $\Ec(P,\xx, \aa)$ \. can be sampled in polynomial time.

Finally, let us mention a curious loop-free listing algorithm in~\cite{CW}.
Is there a similar algorithm for the restricted linear extensions?

\smallskip

\subsection{Injections and matchings for the Stanley inequality}\label{ss:finrem-GP-injection}
As we mentioned in the introduction, it remains a major open problem whether
Stanley's inequality~\eqref{eq:Sta} can be proved by a direct injection,
see e.g.~\cite[$\S$17.17]{CP}.  Formally, in the notation of Theorem~\ref{t:Sta},
let
$$
\rho(P,x,a) \ := \ \aN(P, x,a)^2 \, - \, \aN(P, x,a+1) \.\cdot \.  \aN(P, x,a-1)\ts.
$$

\begin{op} \label{op:Sta-SP}
Is \. $\rho \ts \in \. \SP$?
\end{op}

At this point, it is even hard to guess which way
the answer would go.  While some of us believe the answer should be negative,
others disagree.  The only thing certain is that none of the positive proofs
in~\cite{CP,Sta-AF} imply a positive answer, while the negative results in~\cite{IP}
are not even close to resolving the problem.  Since part of the motivation behind
our algebraic approach aimed at resolving this problem, let us propose the following
approach.

\smallskip

We would like to give an injection proving Stanley's inequality~\eqref{eq:Sta}.
Consider the following family of elements of the group $G$ from Section~\ref{s:restricted} whose actions would be good
candidates for such an injection.
Let\. $\mathcal{G}=(V\sqcup W,E)$ \. be a bipartite graph, where \.
$V=\Ec(P,x,a) \times \Ec(P,x,a)$ \. and \. $W= \Ec(P,x,a-1) \times \Ec(P,x, a+1)$.
We define the set~$E$ of edges as follows.

Let \. $\pi \in \Ec(P,x,a-1)$ \. and \. $\sigma \in \Ec(P,x,a+1)$, so that \. $(\pi,\sigma) \in W$.
For every element \ts $y$ \ts which appears after \ts $x$ \ts in \ts $\pi$ and before \ts $x$ \ts in \ts $\sigma$,
that is \. $i:=\pi^{-1}(y) > a-1$ \. and \. $j:=\sigma^{-1}(y) <a+1$, we apply the
promotion/demotion operators on the chain starting/ending at~$y$ in $\sigma$ and $\pi$, respectively.
Let \. $\bar{\delta}_j := \tau_{n-1}\cdots \tau_j$.  Then in the word \ts $\delta_i \pi$,
the chain starting at~$y$ is pushed up, so that~$x$ is moved to position~$a$. Similarly,
in the word \. $\bar{\delta}_j \sigma$, the chain ending at~$y$ is pushed down,
so that~$x$ is moved to position $a$. Thus \. $(\delta_i \pi, \bar{\delta}_j\sigma) \in \Ec(P;a;x) \times \Ec(P;a;x) = V$,
and we connect it to \ts $(\pi,\sigma)$ \ts by an edge. Note that by the pigeonhole principle,
there are at least two possibilities for elements~$y$, and thus there will be at least one edge,
however it is not necessarily true that the degree of every \ts $(\pi,\sigma)$ \ts is at least~$2$.

\begin{conj}
Let \. $\mathcal{G}=(V\sqcup W, E)$ \. be the graph defined above.
Then there exists a maximal matching which covers all vertices in~$W$.
\end{conj}

This matching will be the desired injection and imply the Stanley inequality.
By itself, the conjecture would not imply that \. $\rho\ts \in\. \SP$\ts.
For that, the injection would need to be computable in polynomial time.

\vskip.5cm

\subsection*{Acknowledgements}
We are grateful to Nikita Gladkov, F\"edor Petrov, Yair Shenfeld and Ramon van Handel
for helpful discussions and remarks on the subject.  Matt Beck, Fu Liu and Sinai Robins
kindly helped us with the Ehrhart polynomial questions.  We thank Christian Gaetz and
Yibo Gao for help with the references, and Darij Grinberg for careful reading of the
paper.  Over the years, we held numerous conversations with Christian Ikenmeyer
on complexity, and the knowledge we acquired has been indispensable.
We  thank the anonymous referees for insightful comments and for additional references.

We thank Vic Reiner for sharing his proof of Theorem~\ref{t:HP-q}
with us and generously allowing us to publish it.  We also thank
Sam Hopkins for telling us about Exc.~3.143 in \cite{Sta-EC} (see
Remark~\ref{r:comp}).  Special thanks to
Richard Stanley for resolving the mystery where Exc.~3.57 comes from,
and telling us about Exc.~3.163 in \cite{Sta-EC}.  We continuously
regret not memorizing all the exercises in~\cite{Sta-EC}.

Finally, we owe a dept of gratitude
to Yufei Zhao who convinced us not to initialize first names
in the references, a practice we followed for years.  It is the right
thing to do and we urge others to follow the suit.\footnote{For more on this, see
Yufei~Zhao,
How I manage my {\tt BibTeX} references, and why I prefer not initializing first names,
\ts \href{https://yufeizhao.com/blog/2021/07/04/bibtex/}{personal blog post} \ts (July 4, 2021).}
This research was partially done while the third author was enjoying
MSRI's hospitality in the Fall of~2021.
The first author was partially supported by the Simons Foundation.
The second and third authors were partially supported by the~NSF.



\vskip.7cm

\end{document}